\documentclass[a4paper, 11pt]{amsart}
\usepackage[english]{babel}
\usepackage[top=1.5in, bottom=.9in, left=0.9in, right=0.9in]{geometry}
\usepackage{graphicx} 
\usepackage{multirow}
\usepackage{tikz, tikz-cd}
\usepackage{mathtools}
\usepackage{xcolor}
\usepackage{thmtools}
\usepackage{amsmath,amssymb,amsthm,color}
\usepackage{verbatim}
\usepackage[colorlinks = true, citecolor=blue, linkcolor = blue, bookmarks=true]{hyperref}
\usepackage[all]{xy}
\usepackage{float}
\declaretheoremstyle[%
  spaceabove=6pt,%
  spacebelow=6pt,%
  headfont=\normalfont\itshape,%
  postheadspace=3pt,%
  qed=\qedsymbol,%
  headpunct={}
]{mystyle} 
\declaretheorem[name={Proof},style=mystyle,unnumbered,
]{Proof}

\theoremstyle{plain}
\newtheorem{teorema}{Theorem} [section]
\newtheorem{proposizione}{Proposition} [section]

\theoremstyle{definition}
\newtheorem{definizione}{Definition}[section]
\newtheorem{esempio}{Example}[section]
\newtheorem{osservazione}{Remark}[section]

\newcommand{\del}{\partial}
\newcommand{\delbar}{\bar{\partial}}
\DeclareMathOperator{\GL}{GL}
\DeclareMathOperator{\SL}{SL}
\DeclareMathOperator{\Ker}{Ker}
\DeclareMathOperator{\im}{Im}
\newcommand{\C}{\mathbb{C}}

\numberwithin{equation}{section}

\title[Hard Lefschetz Condition on symplectic non-K\"ahler solvmanifolds]{Hard Lefschetz Condition on symplectic non-K\"ahler solvmanifolds}
\author{Francesca Lusetti and Adriano Tomassini}
\date{}
\address{Dipartimento di Scienze Matematiche, Fisiche e Informatiche, Unità di Matematica e Informatica\\
Universit\`{a} degli Studi di Parma \\
Parco Area delle Scienze 53/A, 43124 \\
Parma, Italy}
\email{francesca.lusetti2@studenti.unipr.it}
\email{adriano.tomassini@unipr.it}
\thanks{The authors are partially supported by the project PRIN2022 “Real and Complex Manifolds: Geometry and Holomorphic Dynamics” (Project code: 2022AP8HZ9) and by GNSAGA of INdAM}
\begin{document}

\begin{abstract}
We provide new families of compact complex manifolds with no K\"ahler structure carrying symplectic structures satisfying the \textit{Hard Lefschetz Condition}. These examples are obtained as compact quotients of the solvable Lie group $\C^{2n} \ltimes_{\rho} \C^{2m}$, for which we construct explicit lattices. By cohomological computations we prove that such manifolds carry symplectic structures satisfying the \textit{Hard Lefschetz Condition}. Furthermore, we compute the Kodaira dimension of an almost-K\"ahler structure and generators for the de Rham and Dolbeault cohomologies.
\end{abstract}

\maketitle

\tableofcontents

\section{Introduction}
On a $2n$-dimensional symplectic manifold $(M, \omega)$, many operators can be naturally defined. Among them, an important role is played by the \textit{Lefschetz operator} $L$, its formal adjoint $\Lambda$ and the $d^{\Lambda}$ operator. The powers of the Lefschetz operator are isomorphisms between differential forms on $M$ i.e., $A^{n-k}(M) \stackrel{L^{k}}{\simeq} A^{n+k}(M)$. Since $L$ commutes with the exterior derivative $d$, a natural question is whether these maps induce isomorphisms on the de Rham cohomology. If this holds, then $(M, \omega)$ is said to satisfy the \textit{Hard Lefschetz Condition}, or briefly HLC. A classical result states that the fundamental form $\omega$ of a K\"ahler metric $g$ on a compact complex manifold $(M,J)$ satisfies the HLC, namely $(M,\omega)$ is a compact symplectic manifold satisfying the HLC. 
\newline
In \cite{Brylinski} Brylinski, starting with a symplectic manifold $(M,\omega)$, gave the definition of {\em symplectic harmonic} forms, as those $k$-forms which are both $d$ and $d^{\Lambda}$ closed. He formulated the following conjecture: {\em any $k$-de Rham cohomology class contains a symplectic harmonic representative}. He then proved the conjecture for compact K\"ahler manifolds.\newline
Later, O. Mathieu (see \cite{Mathieu}) proved the equivalence between the validity of the Hard Lefschetz Condition and the validity of the \textit{$dd^{\Lambda}$-lemma}, where, by definition, a symplectic manifold is said to satisfy the {$dd^{\Lambda}$-lemma if $$\Ker(d^{\Lambda}) \cap \im(d) = \im(dd^{\Lambda}).$$
Such a definition is a symplectic analogue of the  the\textit{ $\del\delbar$-lemma} for complex manifolds \cite{DGMS}. 
Furthermore, HLC is also related to the symplectic cohomologies $H^{\bullet}_{d+d^\Lambda}(M)$ and $H^{\bullet}_{dd^\Lambda}(M)$ introduced by Tseng and Yau in \cite{TY}, which have remarkable applications in the study of symplectic manifolds.
\newline
A natural question is whether there exist symplectic manifolds with no K\"ahler structure on which HLC holds. Interesting constructions of symplectic structures are obtained on compact quotients of connected and simply-connected nilpotent Lie groups by a lattice, namely on \textit{nilmanifolds}. However, these symplectic manifolds do not satisfy HLC, unless they are tori (see \cite{BG}). In the last years many authors have been studying almost-complex, Hermitian and symplectic geometry of nilmanifolds (see \cite{Ugarte}, \cite{Yan}, \cite{AK1}, \cite{DBT}, \cite{cavalcanti}, \cite{AGS}, \cite{Salamon}, \cite{CFM}, \cite{CFGU}, \cite{UV}, \cite{FGV}, \cite{FV}, \cite{FP} and the references therein). More recent studies have been focused on the wider class of \textit{solvmanifolds}, namely compact quotients of connected and simply-connected solvable Lie groups by a lattice. Such a class includes the one of nilmanifolds but it is way more difficult to build and study. Indeed, not every solvable Lie group admits lattices and explicit lattices are not easy to construct (see for example \cite{Auslander}, \cite{Yamada}); moreover, the de Rham cohomology cannot always be computed through invariant forms (see \cite{Hattori}, \cite{CF}, \cite{Kasuya1}), in contrast to the nilpotent case (see Nomizu's theorem in \cite{Nomizu}). In \cite{Hasegawa} and \cite{Hasegawa1}, K. Hasegawa proved that a compact solvmanifold admits a K\"ahler structure if and only if it is a finite quotient of a complex torus which has a structure of a complex torus bundle over a complex torus and obtained a classification of 4-dimensional compact solvmanifolds admitting complex structures.
Many of the known examples of symplectic solvmanifolds with complex non-K\"ahler structures are 4-dimensional and 6-dimensional (see \cite{TT}, \cite{FGG}, \cite{BG1}, \cite{Ovando}, \cite{DBT1} and the references therein).\\[5pt]
The purpose of this paper is to build new families of compact solvmanifolds admitting complex structures and symplectic structures satisfying the \textit{Hard Lefschetz Condition}, which are not K\"ahler. In particular, our manifolds will be of arbitrary dimension.
We started by examining from a symplectic point of view the \textit{generalized Nakamura manifolds} $N = N_{M,P,\tau}$, introduced by A. Cattaneo and the second author in \cite{CT}, which constitute a generalization of the constructions of I. Nakamura in \cite{N}. Our question is whether such solvmanifolds admit invariant symplectic structures: we show that this is true if and only if a condition on the eigenvalues $\lambda_{1}, ..., \lambda_{n}$ of the matrix $M$ used in the construction holds. In particular, all these symplectic structures over $N$ satisfy the \textit{Hard Lefschetz Condition}. More precisely, we prove the following (see Theorem \ref{SimpletticheHLCGenNakamura}):
\begin{teorema}\label{MainThm1}
    \setlength\itemsep{1em}
    Let $N = N_{M, P, \tau}$ be a generalized Nakamura manifold of complex dimension $n + 1$. 
    \begin{enumerate}
   \item If $n = 2m$, then $N$ admits invariant symplectic structures if and only if there exists a partition $\mathcal{P} = \{(i_1, i_2), ..., (i_{2n-1}, i_{2n})\}$ of the set $\{1, ..., 2n\}$ into disjoint couples such that 
   \[\lambda_i + \lambda_j = 0 \text{ for every } (i, j) \in \mathcal{P}.\]
    \item If $n = 2m+1$ then $N$ admits invariant symplectic structures if and only if there exists a partition $\mathcal{P} = \{(i_1, i_2), ..., (i_{2n-1}, i_{2n}), i_{2n+1}\}$ of the set $\{1, ..., 2n+1\}$ into disjoint couples plus one more index such that 
    \[\lambda_i + \lambda_j = 0 \text{ for every } (i, j) \in \mathcal{P} \text{ and } \lambda_{i_{2n+1}} = 0.\]
    \end{enumerate}
    Moreover, each of these symplectic structures satisfies the Hard Lefschetz Condition.
\end{teorema}
The theorem is proved in section 3, and shows how the \textit{generalized Nakamura manifolds} provide examples of compact solvmanifolds admitting complex and symplectic structures satisfying HLC which are not K\"ahler. A natural question is whether the construction of \textit{generalized Nakamura manifolds} can be somehow generalized to build quotients of solvable Lie groups of type $\C^{N} \ltimes_{\rho} \C^{M}$. The question arises from the papers of H. Kasuya and also D. Angella (see \cite{AK}, \cite{Kasuya}, \cite{Kasuya2}, \cite{Kasuya3}): they show techniques that, under some hypotheses, allow the computation of the de Rham, Dolbeault and Bott-Chern cohomologies of complex manifolds of type $\C^{N} \ltimes_{\varphi} N$, where $N$ is a nilpotent Lie group. In particular, one of the hypothesis is that $\C^{N} \ltimes_{\varphi} N$ admits a lattice which splits as a semi-direct product of a lattice in $\C^{N}$ and one in $N$. As already remarked, effectively building lattices in solvable Lie groups is not an easy task. In section 3 of this paper we explicitly construct families of lattices in $G_{n,m} := \C^{2n} \ltimes_{\rho} \C^{2m}$, where $\rho$ is the action of $\C^{2n}$ on $\C^{2m}$ induced by a unimodular diagonalizable matrix $M \in \SL(2m,\mathbb{Z})$, such that $PMP^{-1} = diag(e^{\lambda}, e^{-\lambda}, ..,e^{\lambda}, e^{-\lambda} )$. We point out that the Lie groups $G_{n,m}$ are 2-step solvable and not of completely solvable type. Hence, Hattori's theorem (see \cite{Hattori}) cannot be applied and it actually turns out that the de Rham cohomology is not the invariant one. We then proceed to study many properties of these solvmanifolds: symplectic structures, HLC, non-K\"ahlerianity, almost-complex structures and their Kodaira dimension.  Our results can be summarized in the following:

\begin{teorema}\label{MainThm2}
    Let $G_{n,m} = \C^{2n} \ltimes_{\rho} \C^{2m}$. Then the following hold:
    \begin{enumerate}
        \item $G_{n,m}$ admits lattices $\Gamma_{\underline{\mu},P}$ and $N_{\underline{\mu},P} = \Gamma_{\underline{\mu},P} \backslash (\C^{2n} \ltimes_{\rho} \C^{2m})$ is a compact solvmanifold admitting complex structures;
        \item $N_{\underline{\mu},P}$ carries symplectic structures satisfying the \textit{Hard Lefschetz Condition};
        \item $N_{\underline{\mu},P}$ carries almost-complex structures which are $\omega$-compatible, not integrable and with zero Kodaira dimension;
        \item $N_{\underline{\mu},P}$ is not K\"ahler.
    \end{enumerate}
\end{teorema}

For the proof of $(1)$ see section 3, in particular Theorem \ref{LatticeInGenNakamura}, while the proofs of $(2), (3)$ and $(4)$ can be also found in section 3, respectively in Theorem \ref{SimpletticheHLCNuove}, Propositions \ref{KodairaNuove} and \ref{NNotKahler}. Moreover, we compute generators for the de Rham and Dolbeault cohomologies making use of the previously introduced techniques and we observe that these manifolds do not satisfy the \textit{$\del\delbar$-lemma}. These computations can be found in section 3, Propositions \ref{GeneratoriDeRhamGenNakamura}, \ref{GeneratoriDolbeaultGenNakamura} and Remark \ref{NoDelDelBar}. The symplectic structures of theorem \ref{MainThm2} are all invariant; however, $N_{\underline{\mu},P}$ also admits non-invariant symplectic structures satisfying HLC. We show an explicit one on $N_{2,2} =\Gamma_{\underline{\mu},P} \backslash (\C^{2} \ltimes_{\rho} \C^{2}) $ in example \ref{ex:3}, but we do not carry out the general case as the amount of computations required rapidly increases with the dimension of $N_{\underline{\mu},P}$.\\[5pt]
The paper is structured as follows. In section 2 we fix some notation recalling definitions and well-known results of which we will make use through the rest of the paper: de Rham and Dolbeault cohomologies, symplectic manifolds, the \textit{Hard Lefschetz Condition} and equivalent properties, solvmanifolds and lattices in solvable Lie groups. Section 3 contains the main results of the paper. We begin by recalling the construction of the \textit{generalized Nakamura manifolds}, followed by the proof of Theorem \ref{MainThm1}. We then proceed with the study of $N_{\underline{\mu},P}$: we show the construction of $\C^{2n} \ltimes_{\rho} \C^{2m}$ and existence of lattices in this solvable Lie group and prove all the statements of our main Theorem \ref{MainThm2}. In section 4 we compute the almost-complex Kodaira dimension of almost-complex structures on the two types of manifolds we considered: we first show that there exist $\omega$-compatible deformations of an almost-K\"ahler structure on the generalized Nakamura manifold $\Gamma_{P,\tau} \backslash (\C \ltimes_{\rho} \C^{4})$ which admit, for certain values of the deformation parameter, almost-complex Kodaira dimension $-\infty$; after, we compute the almost-complex Kodaira dimension of an almost-K\"ahler structure on $N_{\underline{\mu},P}$. Finally, section 5 is devoted to some explicit examples and computations: we apply our results to the quotients of $\C \ltimes_{\rho} \C^{2}$ (which is the \textit{Nakamura manifold}), $\C \ltimes_{\rho} \C^{3}$ (where we show an effective way of seeing the splitting of the lattice), $\C^{2} \ltimes_{\rho} \C^{2}$. Moreover, we show how our assumptions on the eigenvalues of the matrices $M$ are not restrictive, providing in particular two examples (one for $\C \ltimes_{\rho} \C^{3}$ and one for $\C^{2n} \ltimes_{\rho} \C^{4}$).\\[10pt]
\textbf{Acknowledgments} The authors would like to thank Valentino Tosatti for useful comments and remarks. They also want to thank the referee for their accurate observations which led to an improvement of the paper. We are especially thankful for their remarks on the proof of theorem \ref{SimpletticheHLCGenNakamura}.

\section{Preliminaries}
In this section we start by fixing the notation and recalling some definitions and results on \textit{symplectic manifolds}, the \textit{Hard Lefschetz Condition} and the structure of \textit{solvmanifolds}.

\subsection{Notation} The pair $(M,J)$ will denote an almost complex manifold, that is a $2n$-dimensional smooth manifold endowed with a (smooth) $(1,1)$-tensor field $J$ such that $J^2=-id$. The almost complex structure $J$ is said to be integrable if it is induced by the structure of a complex manifold. In view of the \textit{Newlander and Nirenberg Theorem} $J$ is integrable if and only if the Nijenhuis tensor $N_J$ of $J$, defined as
$$
N_J(V,W)=[JV,JW]-[V,W]-J[JV,W]-J[V,JW]
$$
vanishes. In such a case $J$ is a complex structure and $(M,J)$ is a complex manifold. The space of real $k$-forms on $M$, respectively of bi-graded $(p,q)$-forms will be denoted by $A^{k}(M)$, respectively $A^{p,q}(M)$. We recall that
\[A^{k}(M; \mathbb{C}) = \underset{p + q = k}{\bigoplus}A^{p,q}(M)\]
where $A^{k}(M; \mathbb{C})$ denotes the space of complex-valued $k$-forms on $M$ and that on a complex manifold $(M,J)$ the exterior derivative $d$ decomposes as $d = \partial + \overline{\partial}$, where $\partial$ and $\overline{\partial}$ are the projection of $d(A^{k}(M; \mathbb{C}))$ over $A^{p+1,q}(M)$ and $A^{p,q+1}(M)$, respectively.\\ 
We finally fix the notation for the \textit{de Rham} and \textit{Dolbeault} \textit{cohomologies} of a complex manifold $(M,J)$. The \textit{de Rham cohomology group of degree $k$} of $M$ is 
\[ H^{k}_{d}(M) = \dfrac{\Ker(d: A^{k}(M) \rightarrow A^{k+1}(M))}{\im(d: A^{k-1}(M) \rightarrow A^{k}(M))}\]
while the \textit{Dolbeault cohomology group of bi-degree $(p,q)$} is given by
\[H^{p,q}_{\overline{\partial}}(M) = \dfrac{\Ker(\overline{\partial}: A^{p,q}(M) \rightarrow A^{p,q+1}(M))}{\im(\overline{\partial}: A^{p,q-1}(M) \rightarrow A^{p,q}(M))}.\]
We also set
\[H^{*}_{d}(M) = \bigoplus_{k}H^{k}_{d}(M) \qquad H^{*}_{\delbar}(M) = \bigoplus_{p,q}H^{p,q}_{\delbar}(M)\]
and
\[b_{k}(M) = \text{dim}H^{k}_{d}(M) \qquad h_{p,q}(M) = \text{dim}H^{p,q}_{\delbar}(M)\]
for the \textit{Betti numbers} and \textit{Hodge numbers}.

\subsection{Symplectic manifolds and the \textit{Hard Lefschetz Condition}} We will denote by $(M, \omega)$ a symplectic manifold, where $M$ is a $2n$-dimensional manifold and $\omega$ is a non-degenerate, $d$-closed 2-form on $M$. Let $(M,\omega)$ be a symplectic manifold. An almost complex structure $J$ on $M$ is said to be $\omega$-{\em compatible} if, given any $x\in M$, for every $u,v,w\in T_xM$, with $u\neq 0$, the following conditions hold 
$$
\omega_x(u, J_xu) > 0,\qquad
\omega_x(J_xv,J_xw) = \omega_x(v,w).
$$
In such a case, given any pair of vector fields $X,Y\in\Gamma(TM)$, 
$$g_x(X_x,Y_x) = \omega_x(X_x, J_xY_x)$$ defines a $J$-Hermitian metric on $M$.
In particular, every \textit{K\"ahler manifold} is symplectic, with the symplectic form given by the fundamental form of the K\"ahlerian metric and the complex structure is $\omega$-compatible.\\[5pt]
Given a symplectic manifold $(M, \omega)$, the \textit{Lefschetz operator} on $M$ is the operator defined by

\begin{equation}\label{DefLefschetz}
     L : A^{k}(M) \rightarrow A^{k+2}(M) \qquad \alpha \mapsto \omega \wedge \alpha.
\end{equation}
Its dual operator is denoted by
\[ \Lambda: A^{k}(M) \rightarrow A^{k-2}(M). \]
Given these two operators, it is also possible to define the $d^{\Lambda}$ operator, which is given by
\begin{equation}\label{dLambda}
    d^{\Lambda}: A^{k}(M) \rightarrow A^{k-1}(M) \qquad d^{\Lambda} = d\Lambda - \Lambda d. 
\end{equation}
Since $\omega$ is closed, the Lefschetz operator induces a well-defined operator on the de Rham cohomology of $M$ by
\begin{equation}\label{DefLefschetzRham}
        L^{n-k} : H^{k}_{d}(M) \rightarrow H^{2n-k}_{d}(M) \qquad [\alpha] \mapsto [\omega^{n-k} \wedge \alpha].
\end{equation}
The symplectic manifold $(M, \omega)$ is said to satisfy the \textit{Hard Lefschetz Condition} if the Lefschetz operator on the de Rham cohomology of $M$ is an isomorphism for all $k = 0, ..., n$. Moreover, $(M, \omega)$ is said to satisfy the \textit{$dd^{\Lambda}$-lemma} if 
\begin{equation*}
    \Ker(d^{\Lambda}) \cap \im(d) = \im(dd^{\Lambda}).
\end{equation*}
In particular, the following facts are equivalent on a compact symplectic
manifold $(M^{2n},\omega)$ (see \cite{Brylinski}, \cite{Mathieu}, \cite{merkulov},
\cite{Yan}, \cite{cavalcanti})
\begin{itemize}
\item{} the Hard Lefschetz Condition  holds;
\item{} the Brylinski conjecture holds;
\item{} the \emph{$dd^\Lambda$-lemma} holds;
\item{} the natural maps induced by the identity $H^{\bullet}_{d+d^\Lambda}(M)
\longrightarrow H^{\bullet}_{d}(M)$ are injective;
\item{} the natural maps induced by the identity $H^{\bullet}_{d+d^\Lambda}(M)
\longrightarrow H^{\bullet}_{d}(M)$ are surjective;
\item{} the natural maps induced by the identity in the following diagram are isomorphisms
$$ \xymatrix{
  & H^{\bullet}_{d+d^\Lambda}(M) \ar[ld]\ar[rd] & \\
  H^{\bullet}_{d}(M) \ar[rd] &  & H^{\bullet}_{d^\Lambda}(M). \ar[ld] \\
  & {\phantom{\;.}} H^{\bullet}_{dd^\Lambda}(M) \; &
} $$
\end{itemize}

where 
\[H^{k}_{d+d^{\Lambda}}(M) = \dfrac{\ker d+d^{\Lambda} \cap A^{k}(M)}{\im dd^{\Lambda} \cap A^{k}(M)} \qquad H^{k}_{dd^{\Lambda}}(M) = \dfrac{\ker dd^{\Lambda} \cap A^{k}(M)}{(\im d+\im d^{\Lambda}) \cap A^{k}(M)}\]
are the symplectic cohomologies introduced by Tseng and Yau in \cite{TY}. 

%
%
%
%
%

\subsection{Solvmanifolds} We now recall the definition of \textit{solvmanifold} and some properties of \textit{lattices} in solvable Lie groups. 
By a {\em solvmanifold} we mean a compact quotient $M=\Gamma \backslash G$ of a connected and simply-connected solvable Lie group $G$ by a lattice $\Gamma$. In particular, $\pi_{1}(\Gamma \backslash G) \simeq \Gamma$. 
We recall two results, the first one due to Mostow, describing the structure of solvmanifolds whose Lie groups $G$ are simply connected. Both theorems can be found in \cite[Theorem 3.6 and Corollary to Proposition 3.7]{Raghunathan}.

\begin{teorema}\label{MostowFundGroups}
    Let $G_{1}$, $G_{2}$ be two simply connected solvable Lie groups and let $\Gamma_{1}$, $\Gamma_{2}$ be lattices in $G_{1}$ and $G_{2}$, respectively. If $\Gamma_{1}$ and $\Gamma_{2}$ are isomorphic then the corresponding solvmanifolds $\Gamma_{1} \backslash G_{1}$ and $\Gamma_{2} \backslash G_{2}$ are diffeomorphic.
\end{teorema}
\begin{teorema}
    A lattice $\Gamma$ in a connected solvable Lie group $G$ is finitely generated.
\end{teorema}
We recall some results for the computation of the de Rham cohomology of solvmanifolds. Let $\Gamma \backslash G$ be a solvmanifold and let $\mathfrak{g}$ be the Lie algebra of $G$. It is well-known that, differently from \textit{nilmanifolds}, it is not true in general that $H^{*}(\Gamma \backslash G) \simeq H^{*}(\mathfrak{g})$. For solvmanifolds this is true if and only $\Gamma \backslash G$ satisfies the \textit{Mostow condition}. In \cite{Hattori}, A. Hattori proved that this holds in the case where $G$ is of \textit{completely solvable type} i.e., the adjoint action of $\mathfrak{g}$ on itself has only real eigenvalues. In the general case, the cohomologies of a solvmanifold $\Gamma \backslash G$ cannot be computed using only invariant forms. There are however some useful results for computing the de Rham and Dolbeault cohomologies of solvmanifolds due to H. Kasuya which we will later use and can be found in \cite{Kasuya}.

\subsection{Kodaira dimension of almost-complex manifolds} Let $(M, J)$ be an almost-complex manifold of dimension $n$. The \textit{Kodaira dimension} $\kappa^{J}(M)$ of $(M, J)$ was defined by H. Chen and W. Zhang in \cite{CZ} (see also \cite{CZ1} and \cite{CNT}, \cite{CNT1} and \cite{CNT2}). They defined $\kappa^{J}(M)$ as:
\[\kappa^{J}(M) = \begin{cases}
    - \infty & \text{if } P_{r, \mathcal{J}} = 0 \quad \forall r \ge 0 \\
    \limsup_{r \rightarrow +\infty}\dfrac{\log P_{r,\mathcal{J}}}{\log r} & \text{otherwise }\\
\end{cases}\]
where $P_{r, \mathcal{J}} := \text{dim}H^{0}(M, (\Lambda^{n,0}(M))^{\otimes r})$ are the \textit{plurigenera} of $(M, J)$.

\section{Non-K\"ahler solvmanifolds satisfying the Hard Lefschetz Condition}
In this section we prove the main theorems stated in the introduction. We begin by recalling the construction of the \textit{generalized Nakamura manifolds} as in \cite{CT} and proving that they carry symplectic structures satisfying the \textit{Hard Lefschetz Condition}. Then, we construct the quotients of $\C^{2n} \ltimes_{\rho} \C^{2m}$ and we prove all the claims of our main Theorem \ref{MainThm2}.

\subsection{Hard Lefschetz Condition on generalized Nakamura manifolds}
We recall the construction of the \textit{generalized Nakamura manifolds} provided in \cite{CT}. We show that, under some new hypotheses, that these manifolds carry invariant symplectic structures. \\
Let $M \in \SL(n,\mathbb{Z})$ be a matrix with positive eigenvalues $\{e^{\lambda_{1}}, ..., e^{\lambda_{n}}\}$ such that
\begin{equation}
	PMP^{-1} = \hbox{diag}(e^{\lambda_{1}}, ..., e^{\lambda_{n}}) =: D
\end{equation}
with $P \in \GL(n,\mathbb{R})$. Since $M \in \SL(n, \mathbb{Z})$, we have
\[\sum_{i = 1}^n \lambda_i = 0.\]
Define the action of $\C$ over $\C^{n}$ given by
\[\rho : \mathbb{C} \rightarrow GL(n, \mathbb{C}) \qquad w \longmapsto \hbox{diag}(e^{\frac{1}{2}\lambda_{1}(w + \overline{w})}, e^{\frac{1}{2}\lambda_{2}(w + \overline{w})}, ..., e^{\frac{1}{2}\lambda_{n}(w + \overline{w})}).\]
Consider now the well-defined Lie group (not complex)
\[  (G_{M}, *) := \mathbb{C} \ltimes_{\rho} \mathbb{C}^{n} \]
where $*$ is explicitly given by
\[ (w', z'_{1}, \ldots, z'_{n}) * (w, z_{1}, \ldots, z_{n}) = (w' + w, \rho(w')(z_{1}, \ldots, z_{n}) + (z'_{1}, \ldots, z'_{n})). \]
As was proved in \cite[subsection 4.5]{CT}, $(G_{M}, *)$ is a 2-step solvable non-nilpotent Lie group. Moreover, it is of \textit{completely solvable type}, as the adjoint action of $\mathfrak{g}$ on itself results in a lower triangular matrix.
We now recall how to obtain compact quotients of $(G_{M}, *)$. Fix $\tau \in \mathbb{R} \setminus \{0\}$ and consider the following \textit{lattices} in $\mathbb{C}$ and $\mathbb{C}^{n}$, respectively:
\begin{equation}\label{LatticesInNakamura}
    \Gamma'_{\tau} := \mathbb{Z} \oplus \tau\sqrt{-1}\cdot\mathbb{Z} \qquad
    \Gamma''_{P} := P\mathbb{Z}^{n} \oplus \sqrt{-1}P\mathbb{Z}^{n}.
\end{equation}
Then
\[ \Gamma_{P, \tau} :=  \Gamma'_{\tau} \ltimes_{\rho} \Gamma''_{P}\]
is a lattice in $(G_{M}, *)$ which makes
\begin{equation}\label{DefinizioneGenNakamura}
N_{M, P, \tau} = \Gamma_{P, \tau} \backslash (\mathbb{C} \ltimes_{\rho} \mathbb{C}^{n})
\end{equation}
a solvmanifold}, called the \textit{Generalized Nakamura manifold} associated to the triple $(M, P, \tau)$.

\begin{osservazione}\label{LatticeSpezza}
    The case $n = 2m+1$, $\lambda_{2m+1} = 0$ corresponds to a translation of the last coordinate of $\C^{2m+1}$. Then $$\Gamma_{P, \tau} \simeq (\Gamma'_{\tau} \ltimes_{\rho} \Gamma''_{Q}) \times \Gamma^{'''}_{a,b}$$ as groups, where $\Gamma^{'''}_{a,b} = a\mathbb{Z} \oplus \sqrt{-1}\cdot b\mathbb{Z}$ for some $a,b \in \mathbb{R}\setminus \{0\}$. Thus, the corresponding manifold $N$ splits as
    \[N \simeq N^{'} \times \mathbb{T}^{1}_{\mathbb{C}}\]
    where $N'$ is a generalized Nakamura manifold of dimension $n = 2m$ and $\mathbb{T}^{1}_{\mathbb{C}}$ is a complex torus. 
\end{osservazione}
The forms
\begin{equation}
    \varphi^0 = dw\, \qquad \varphi^i = e^{-\frac{1}{2}\lambda_i (w + \bar{w})} dz_i \qquad i = 1, \ldots, n
\end{equation}
give a well-defined set of $(1,0)$-forms on $N_{M, P, \tau}$. A straightforward calculation then shows that the following structure equations hold:
\begin{equation}\label{StrEqn}
    d\varphi^{0} = 0 \qquad d\varphi^{i} = -\dfrac{1}{2}\lambda_{i}(\varphi^{0} + \bar\varphi^{0}) \wedge \varphi^{i} \quad i = 1, ..., n.
\end{equation}
The structure equations also show that $N_{M, P, \tau}$ is a compact complex manifold which is not holomorphically parallelizable.
\begin{osservazione}
    If $\lambda_{i} = 0 $ $\forall i $, then the manifold $N$ is just a complex torus. Starting from theorem \ref{SimpletticheHLCGenNakamura}, we will assume there exists at least one index $j$ such that $\lambda_{j} \ne 0$ in order to exclude the torus, as in \cite{N} and \cite{CT}.
\end{osservazione}
We now prove that the manifolds $N = N_{M, P, \tau}$ carry invariant symplectic structures. Moreover, each one of these symplectic structures satisfies the \textit{Hard Lefschetz Condition}. If the action $\rho$ is non trivial i.e., there exists at least one $\lambda_{i} \ne 0$, the generalized Nakamura manifolds $N_{M, P, \tau}$ provide families of compact complex symplectic non-K\"ahler solvmanifolds satisfying the Hard Lefschetz condition and the $\del\delbar$-lemma, as was proved in \cite[Theorem 4.14]{CT}. We are ready to prove Theorem \ref{MainThm1}:
\begin{teorema}\label{SimpletticheHLCGenNakamura}
    \setlength\itemsep{1em}
    Let $N = N_{M, P, \tau}$ be a generalized Nakamura manifold of complex dimension $n + 1$. 
    \begin{enumerate}
   \item If $n = 2m$, then $N$ admits invariant symplectic structures if and only if there exists a partition $\mathcal{P} = \{(i_1, i_2), ..., (i_{2m-1}, i_{2m})\}$ of the set $\{1, ..., 2m\}$ into disjoint pairs such that 
   \[\lambda_i + \lambda_j = 0 \text{ for every } (i, j) \in \mathcal{P}.\]
    \item If $n = 2m+1$ then $N$ admits invariant symplectic structures if and only if there exists a partition $\mathcal{P} = \{(i_1, i_2), ..., (i_{2m-1}, i_{2m}), i_{2m+1}\}$ of the set $\{1, ..., 2m+1\}$ into disjoint pairs plus one more index such that 
    \[\lambda_i + \lambda_j = 0 \text{ for every } (i, j) \in \mathcal{P} \text{ and } \lambda_{i_{2m+1}} = 0.\]
    \end{enumerate}
    Moreover, each of these symplectic structures satisfies the Hard Lefschetz Condition.
\end{teorema}
Before giving the proof of Theorem \ref{SimpletticheHLCGenNakamura}, we need to recall the structure of the de Rham cohomology of $N$. From now on $I, J$ will denote multiindices in $\{1, \ldots,n\}$, and we set 
\[c_{IJ} = \sum_{k \in I \cup J} \lambda_{k}.\]
Since $N$ is of completely solvable type, we have the isomorphism
\[H^{*}_d(N) \simeq H^{*}_d(\mathfrak{g})\]
where $\mathfrak{g}$ is the Lie algebra of the Lie group $G_{M}$. Then the de Rham cohomology of $N$ can be computed just by considering the \textit{invariant} forms on $N$. For the explicit computation of generators of $H^{k}_{d}(N)$ (see \cite[Proposition 4.12]{CT}) we recall that $H^{k}_{d}(N)$, $k = 1 ,\ldots,2n$ is generated by

     \begin{itemize}
        \item $\varphi^I \wedge \bar{\varphi}^J$ such that $c_{IJ} = 0$ and $|I| + |J| = k$;\\[1pt]
        \item $\varphi^0 \wedge \varphi^I \wedge \bar{\varphi}^J$ such that $c_{IJ} = 0$ and $|I| + |J| = k - 1$;\\[1pt]
        \item $\bar\varphi^0 \wedge \varphi^I \wedge \bar{\varphi}^J$ such that $c_{IJ} = 0$ and $|I| + |J| = k - 1$.
    \end{itemize}

\begin{Proof}\textit{of Theorem \ref{SimpletticheHLCGenNakamura}}
    We start by proving (1), i.e., $n = 2m$. The case (2), corresponding to $n = 2m+1$, can be proved in the same way just by making an observation about the de Rham cohomology. In the entire proof we make use of the complex formalism, considering the cohomology groups $H_{d}^{k}(N;\C)$.\\[5pt]
    Let $n = 2m$ and suppose there exists the partition $\mathcal{P} = \{(i_1, i_2), ..., (i_{2m-1}, i_{2m})\}$. Then every $2$-form given by
    \begin{equation}\label{FormaSimplDimDispari}
        \omega = \dfrac{i}{2}C\varphi^{0} \wedge \bar\varphi^{0} + \dfrac{1}{2}\sum_{(i,j) \in \mathcal{P}}^{}\left(A_{ij}\varphi^{i} \wedge \varphi^{j} + B_{ij}\varphi^{i} \wedge \bar\varphi^{j} + \bar B_{ij}\bar\varphi^{i} \wedge \varphi^{j} + \bar A_{ij}\bar\varphi^{i} \wedge \bar\varphi^{j}\right)
    \end{equation}
such that 
\begin{equation}\label{conditions-coefficients}
C \in \mathbb{R} \setminus \{0\}, (A_{ij}, B_{ij}) \in \mathbb{C}^{2} \setminus \{(0,0)\}\qquad \hbox{for all}\,\,\, (i,j) \in \mathcal{P}
\end{equation}
is a symplectic form on $N$: indeed, by a direct computation one can see that the conditions \eqref{conditions-coefficients} on the coefficients $C, A_{i}, B_{i}$ and the indices in $\mathcal{P}$ imply that the above $2$-form $\omega$ is non degenerate. The structure equations \eqref{StrEqn} imply that $d\omega=0$.
Indeed, we get:
\[d(\varphi^{i} \wedge \varphi^{j}) = -\frac{1}{2}\lambda_{i}(\varphi^{0} + \overline{\varphi}^{0}) \wedge \varphi^{i} \wedge \varphi^{j} - \frac{1}{2}\lambda_{j} (\varphi^{0} + \overline{\varphi}^{0})\varphi^{i} \wedge \varphi^{j} = 0\]
since $\lambda_i + \lambda_j = 0$. By a similar calculation, the other summands in \eqref{FormaSimplDimDispari} are $d$-closed. Therefore, for any given $C, A_{i}, B_{i}$ as in \eqref{conditions-coefficients}, $\omega$ defined in \eqref{FormaSimplDimDispari} is a symplectic structure on $N$. 
\vskip.2truecm\noindent
Conversely, assume that for every partition $\mathcal{P} = \{(i_1, i_2), ..., (i_{2m-1}, i_{2m})\}$ there exists at least one couple $(i,j) \in \mathcal{P}$ such that $\lambda_i + \lambda_j \ne 0$. We recall that the real (complex) de Rham cohomology of $N$ can be computed by invariant forms: in particular, we have\vskip.2truecm\noindent
\begin{equation}\label{2-nd-cohomology-N} 
    H^{2}_{d}(N;\C) \simeq \mathbb{C}\langle \varphi^{0} \wedge \bar\varphi^{0}, \varphi^{i_{1}} \wedge \varphi^{i_{2}} \text{ with } \lambda_{i_{1}} + \lambda_{i_{2}} = 0, \varphi^{i} \wedge \bar\varphi^{j} \text{ with } \lambda_{i} + \lambda_{j} = 0, \bar\varphi^{j_{1}} \wedge \bar\varphi^{j_{2}} \text{ with } \lambda_{j_{1}} + \lambda_{j_{2}} = 0 \rangle. 
\end{equation}
Recall that a symplectic form $\omega$ must lie in $H^{2}_d(N;\mathbb{R})$ and must be non-degenerate. Then in the indices appearing in its expression there must appear at least one partition of $\{1,..., 2m\}$ into disjoint pairs $(i,j)$ such that $\lambda_i + \lambda_j = 0$. However, in this case such a partition cannot exist therefore, there are no symplectic structures $\omega$.\\[5pt]
    If $n = 2m +1$ (case (2)), by the same arguments as before, one can show that a symplectic form on $N$ is given by 
    \begin{gather*}\label{FormaSimplDimPari}
             \omega = \dfrac{i}{2}C\varphi^{0} \wedge \bar\varphi^{0} + \dfrac{1}{2}\sum_{(i,j) \in \mathcal{P}}^{}(A_{ij}\varphi^{i} \wedge \varphi^{j} + B_{ij}\varphi^{i} \wedge \bar\varphi^{j} + \bar B_{ij}\bar\varphi^{i} \wedge \varphi^{j} + \bar A_{ij}\bar\varphi^{i} \wedge \bar\varphi^{j} + \\
            + \dfrac{i}{2}D\varphi^{i_{2m+1}} \wedge \bar\varphi^{i_{2m+1}})
    \end{gather*}
    where $C, A_{ij}, B_{ij}$ are as in \eqref{conditions-coefficients} and $D \in \mathbb{R} \setminus \{0\}$. We point out that the condition $$\lambda_{i_{2m+1}} = 0$$ is necessary in order that 
    $$\varphi^{0} \wedge \bar\varphi^{i_{2m+1}}, \quad \varphi^{i_{2m+1}} \wedge \bar\varphi^{0} \quad \text{and}\quad \varphi^{i_{2m+1}} \wedge \bar\varphi^{i_{2m+1}}$$
    are $d$-closed. Therefore, such a condition ensures the existence of non-degenerate $d$-closed 2-forms, looking at $H^2_d(N; \mathbb{R})$ and following the same argument as before.\\[5pt]
    From remark \ref{LatticeSpezza}, it is now sufficient to work with $n = 2m$. The $k$-power of $\omega$ has the form
    \[\omega^{k} = k!\bigg(\sum_{\substack{I,J \\ |I| + |J| = 2k-2 \\ 
    c_{IJ} = 0 \\ 0 \notin I \cup J}} A_{IJ}\varphi^{0} \wedge \bar\varphi^{0} \wedge \varphi^{I} \wedge \bar\varphi^{J} +\sum_{\substack{I,J \\ |I| + |J| = 2k \\ 
    c_{IJ} = 0 \\ 0 \notin I \cup J}} B_{IJ}\varphi^{I} \wedge \bar\varphi^{J}\bigg).  \]
    Then we get for $L^{k}: H^{2m+1-k}_{d}(N) \rightarrow H^{2m+1+k}_{d}(N)$:
    \begin{equation*}
    \begin{split}
        &\omega^{k} \wedge \varphi^I \wedge \bar{\varphi}^J = k!\bigg(\sum_{\substack{|I|+|J| = 2m+1-k \\ c_{IJ} = 0 \\I' \cap I = \emptyset \\ J' \cap J = \emptyset}} A^{'}_{I'J'} \varphi^{0} \wedge \bar\varphi^{0} \wedge \varphi^{I' \cup I} \wedge \bar\varphi^{J' \cup J} +\sum_{\substack{|I|+|J| = 2m+1-k \\ c_{IJ} = 0 \\ I' \cap I = \emptyset \\ J' \cap J = \emptyset}} B^{'}_{I'J'}\varphi^{I' \cup I} \wedge \bar\varphi^{J' \cup J}\bigg) \\
        &\omega^{k} \wedge \varphi^{0} \wedge \varphi^I \wedge \bar{\varphi}^J = k!\bigg(\sum_{\substack{|I|+|J| = 2m-k \\ c_{IJ} = 0 \\I' \cap I = \emptyset \\ J' \cap J = \emptyset}} B^{''}_{I'J'}\varphi^{0} \wedge \varphi^{I' \cup I} \wedge \bar\varphi^{J' \cup J}\bigg)\\
        &\omega^{k} \wedge \bar\varphi^{0} \wedge \varphi^I \wedge \bar{\varphi}^J = k!\bigg(\sum_{\substack{|I|+|J| = 2m-k \\ c_{IJ} = 0 \\I' \cap I = \emptyset \\ J' \cap J = \emptyset}} B^{'''}_{I'J'}\bar\varphi^{0} \wedge \varphi^{I' \cup I} \wedge \bar\varphi^{J' \cup J}\bigg)
    \end{split}
    \end{equation*}
    which are all independent generators of $H^{2m+1+k}_{d}(N)$ as a consequence of the non-degeneracy of $\omega$ and the conditions over the multiindices. 
\end{Proof}

\subsection{Hard Lefschetz Condition on quotients of $\C^{2n} \ltimes_{\rho} \C^{2m}$}
We now construct new families of compact complex symplectic non-K\"ahler solvmanifolds as quotients of Lie groups (not complex) of the form $\mathbb{C}^{2n} \ltimes_{\rho} \mathbb{C}^{2m}$. These solvmanifolds are examples of the constructions introduced by H. Kasuya in \cite{Kasuya}. We then prove all the statements of Theorem \ref{MainThm2}.

\subsubsection{Construction} Let $M \in \SL(2m, \mathbb{Z})$ be a matrix for which there exists $P \in \GL(2m, \mathbb{R})$ such that 
\[PMP^{-1} = \hbox{diag}(e^{\lambda}, e^{-\lambda}, \ldots, e^{\lambda}, e^{-\lambda}) =: D \]
with $\lambda \ne 0$ (we will show through examples that this condition is also non-empty and non-trivial). Again, the assumption on the eigenvalues of $M$ is to be considered up to possible rearrangements of them.\\[5pt]
Denote the coordinates of $\mathbb{C}^{2n}$ with $w := (w_{1}, \ldots, w_{2n})$ and the ones of $\mathbb{C}^{2m}$ with $z := (z_{1}, \ldots, z_{2m})$. We define a group action of $\mathbb{C}^{2n}$ over $\mathbb{C}^{2m}$ by setting $\rho : \mathbb{C}^{2n} \rightarrow \GL(2m, \mathbb{C})$
\[w = (w_{1}, \ldots, w_{2n}) \mapsto \rho(w) = \hbox{diag}(e^{\lambda(w_{1} + \bar w_{2} + \cdots + w_{2n-1} + \bar w_{2n})}, e^{-\lambda(\bar w_{1} + w_{2} + \cdots + \bar w_{2n-1} + w_{2n})}, \ldots, \]
\[ \ldots, e^{\lambda(w_{1} + \bar w_{2} + \cdots + w_{2n-1} + \bar w_{2n})}, e^{-\lambda(\bar w_{1} + w_{2} + \cdots + \bar w_{2n-1} + w_{2n})}). \]
Let $G_{n,m} := \mathbb{C}^{2n} \ltimes_{\rho} \mathbb{C}^{2m}$. We point out that $G_{n,m}$ has the form $\mathbb{C}^{N} \ltimes_{\rho} N$, where $N = \mathbb{C}^{2m}$ is a nilpotent Lie group with a left-invariant complex structure. Let $\mathfrak{g} = \text{Lie}(G_{n,m})$ be the Lie algebra of $G_{n,m}$. A basis for $\mathfrak{g}_{\mathbb{C}} = \mathfrak{g} \otimes \mathbb{C}$ is given by the left-invariant vector fields
\[ W_{i} = \frac{\del}{\del w_{i}} \quad i = 1, \ldots, 2n \]
\[Z_{2j+1} = e^{\lambda(w_{1} + \bar w_{2} + \cdots + w_{2n-1} + \bar w_{2n})}\frac{\del}{\del z_{2j+1}} \quad j = 0, \ldots, m-1 \]
\[Z_{2j+2} = e^{-\lambda(\bar w_{1} + w_{2} + \cdots + \bar w_{2n-1} + w_{2n})}\frac{\del}{\del z_{2j+2}} \quad j = 0, \ldots, m-1 \]
and their complex conjugates. By setting $W_{i} =: E_{i} + \sqrt{-1}F_{i}$, $Z_{j} =: e_{j} + \sqrt{-1}f_{j} $, we get a basis for $\mathfrak{g}$, with non trivial commutators given by

\begin{gather*}\label{CommutatoriNonCSolvable}
    [E_{i}, e_{2j+1}] = \lambda e_{2j+1} \qquad [E_{i}, f_{2j+1}] = \lambda f_{2j+1} \\
    [F_{i}, e_{2j+1}] = (-1)^{i}\lambda f_{2j+1} \qquad [F_{i}, f_{2j+1}] = (-1)^{i}\lambda e_{2j+1} \\
    [E_{i}, e_{2j+2}] = -\lambda e_{2j+2} \qquad [E_{i}, f_{2j+2}] = -\lambda f_{2j+2} \\
    [F_{i}, e_{2j+2}] = (-1)^{i+1}\lambda f_{2j+2} \qquad [F_{i}, f_{2j+2}] = (-1)^{i+1}\lambda e_{2j+2}. 
\end{gather*}
Thus, we get:
\[\mathfrak{g}, \mathfrak{g}] = \text{span}_{\mathbb{R}}\{e_{j}, f_{j}\} \qquad [[\mathfrak{g}, \mathfrak{g}],[\mathfrak{g}, \mathfrak{g}]] = \{0\} \qquad [\mathfrak{g},[\mathfrak{g}, \mathfrak{g}]] = [\mathfrak{g}, \mathfrak{g}]\]
which show that $G_{n,m}$ is a 2-step solvable but not nilpotent Lie group. However, $G_{n,m}$ is not of completely solvable type, as we also get pairs of complex conjugate eigenvalues for $Ad_{\mathfrak{g}}$.\\[5pt]
We now show how we can build compact quotients of $G_{n,m}$. Consider the following lattices in $\mathbb{C}^{2n}$ and $\mathbb{C}^{2n}$, respectively, defined by
\[ \Gamma'_{\underline{\mu}} := \mathbb{Z}^{2n} \oplus \sqrt{-1}\cdot \underline{\mu}\mathbb{Z}^{2n} \qquad
    \Gamma''_{P} := P\mathbb{Z}^{2m} \oplus \sqrt{-1}P\mathbb{Z}^{2m}\]
where $\underline{\mu} := \hbox{diag}(\mu_{1},\ldots, \mu_{2n})$ and $\underline{\mu}\mathbb{Z}^{2n}$ stands for vectors of the form $$\underline{\mu}\,^t(m_{1}, \ldots, m_{2n})=^t\!\!(\mu_{1}m_{1},\ldots, \mu_{2n}m_{2n})$$ where $^t(m_{1}, \ldots, m_{2n}) \in \mathbb{Z}^{2n}$. We prove the following:
\begin{teorema}\label{LatticeInGenNakamura}
    For $\mu_{i} = \frac{2k_{i}\pi}{\lambda}$, $k_{i} \in \mathbb{Z}$, $\Gamma_{\underline{\mu}, P} :=  \Gamma'_{\underline{\mu}} \ltimes_{\rho} \Gamma''_{P}$ is a lattice in $G_{n,m}$.
\end{teorema}
\begin{proof}
    Since $\rho((0,\ldots, 1, \ldots, 0)) = D$ and $DP = PM$ with $M \in \SL(2m,\mathbb{Z})$, the action of $\mathbb{Z}^{2n}$ preserves $\Gamma''_{P}$. Since we have
    \[\rho(\sqrt{-1}(0,\ldots, \mu_{i}, \ldots, 0)) = \hbox{diag}(e^{(-1)^{i-1}\sqrt{-1}\lambda \mu_{i}}, e^{(-1)^{i}\sqrt{-1}\lambda \mu_{i}}, \ldots, e^{(-1)^{i-1}\sqrt{-1}\lambda \mu_{i}}, e^{(-1)^{i}\sqrt{-1}\lambda \mu_{i}}) = \mathbb{I}_{2m}\]
    for the choice of $\mu_{i}$, also the action of $\sqrt{-1}\cdot \underline{\mu}\mathbb{Z}^{2n}$ preserves $\Gamma''_{P}$. Then $\Gamma_{\underline{\mu}, P} :=  \Gamma'_{\underline{\mu}} \ltimes_{\rho} \Gamma''_{P}$ is a discrete subgroup of $G_{n,m}$ which is also a well-defined lattice.
\end{proof}
 Now we set
 \begin{equation}\label{DefGenNakamura}
     N := N_{\underline{\mu}, P} = \Gamma_{\underline{\mu}, P}  \backslash (\mathbb{C}^{2n} \ltimes_{\rho} \mathbb{C}^{2m}).
 \end{equation}
Then $N$ is a compact solvmanifold which is not of completely solvable type.

\bigskip
\noindent
A global co-frame of (1,0)-forms on $N$ is given by
\begin{equation*}\label{FormeNakamuraNuove}
    \varphi^{i} = dw_{i} \qquad \psi^{2j+1} = e^{-\lambda(w_{1} + \bar w_{2} + \cdots + w_{2n-1} + \bar w_{2n})}dz_{2j+1} \qquad \psi^{2j+2} = e^{\lambda(\bar w_{1} + w_{2} + \cdots + \bar w_{2n-1} + w_{2n})}dz_{2j+2}
\end{equation*}
for $i = 1, \ldots, 2n$ and $j = 0, \ldots, m-1$. Setting $$\eta := \varphi^{1} + \bar\varphi^{2} + \cdots + \bar \varphi^{2n}, $$ we get the following structure equations:
\begin{equation}\label{EqStrNakamuraNuove}
    d\varphi^{i} = 0 \qquad d\psi^{2j+1} = -\lambda\eta \wedge \psi^{2j+1} \qquad d\psi^{2j+2} = \lambda\bar\eta \wedge \psi^{2j+2}.
\end{equation}
Thus, the co-frame $\{\varphi^{i}, \psi^{j}\}$ defines a complex structure on $N$, which makes $N$ a compact  solvmanifold endowed with a complex structure.

\subsubsection{de Rham and Dolbeault cohomologies} We compute generators for the de Rham and Dolbeault cohomology groups of $N = N_{\underline{\mu}, P}$. We will make use of these results to show that $N$ carries symplectic structures satisfying the Hard Lefschetz Condition but does not satisfy the $\del\delbar$-lemma. Since $N$ is not of completely solvable type, we recall that it is not guaranteed that the de Rham cohomology will coincide with the one of left-invariant forms. In this case, both the de Rham and Dolbeault cohomologies are not invariant. For our computations we make use of the results obtained by H. Kasuya in \cite{Kasuya2} and \cite{Kasuya}.\\[5pt]
Starting from the de Rham cohomology, we remark that $G := G_{n,m} = \mathbb{C}^{2n} \ltimes_{\rho} \mathbb{C}^{2m}$ is a simply connected Lie group with Lie algebra $\mathfrak{g}$ such that $\text{dim}\mathfrak{g}_{\C} = 4n+4m$. Let $\phi$ be the extension to $G$ of the action of the semi-simple part of $\mathfrak{g}_{\C}$ on itself. Since this part corresponds to the Lie algebra of the factor $\C^{2n}$, $\phi$ is represented by the matrix
\[\phi = \hbox{diag}(1,\ldots, 1, e^{\lambda(w_{1} + \bar w_{2} + \cdots + w_{2n-1} + \bar w_{2n})}, e^{\lambda(\bar w_{1} + w_{2} + \cdots + \bar w_{2n-1} + w_{2n})}, \ldots, e^{-\lambda(w_{1} + \bar w_{2} + \cdots + \bar w_{2n-1} + \bar w_{2n})}).\]
This means that $\phi$ is diagonal, then it is semi-simple and the characters $\{\alpha_{1}, \cdots, \alpha_{4n+4m}\}$ defined by Kasuya are given by:
\begin{itemize}
    \item $\alpha_{1} = \cdots = \alpha_{4n} = 1$,\\[1pt]
    \item $\alpha_{4n+4i+1} = e^{\lambda(w_{1} + \bar w_{2} + \cdots + w_{2n-1} + \bar w_{2n})}$, $\alpha_{4n+4i+2} = e^{ \lambda(\bar w_{1} + w_{2} + \cdots + \bar w_{2n-1} + w_{2n})}$,\\[1pt]
    \item $\alpha_{4n+4i+3} = e^{-\lambda(\bar w_{1} + w_{2} + \cdots + \bar w_{2n-1} + w_{2n})}$, $\alpha_{4n+4i+4} = e^{-\lambda(w_{1} + \bar w_{2} + \cdots + w_{2n-1} + \bar w_{2n})}$,
\end{itemize}
for $i = 0, \ldots, m-1$. We then see that the products of type $\alpha_{i_{1}} \cdots \alpha_{i_{p}}$ which are trivial on $\Gamma_{\underline{\mu}, P} $ are combinations of:
\begin{itemize}
    \item $\alpha_{1}, \ldots,\alpha_{4n}$\\[1pt]
    \item $\alpha_{4n+4i+1}\cdot\alpha_{4n+4i+4} = \alpha_{4n+4i+2}\cdot\alpha_{4n+4i+3} = 1$\\[1pt]
    \item $\alpha_{4n+4i+1}\cdot\alpha_{4n+4i+3} = e^{\lambda(w_{1} -\bar w_{1} + \cdots + \bar w_{2n} - w_{2n})}$\\[1pt]
    \item $\alpha_{4n+4i+2}\cdot\alpha_{4n+4i+4} = e^{\lambda(-w_{1} +\bar w_{1} + \cdots  -\bar w_{2n} + w_{2n})}.$
\end{itemize}
From this calculation it is straightforward to see that $d|_{A^{*}_{\Gamma_{\underline{\mu}, P} }} = 0$, where $A^{*}_{\Gamma_{\underline{\mu}, P} }$ is the sub-DGA of $A^{*}(N; \C)$ as in \cite[lemma 2.6]{Kasuya2}:
\[A^{p}_{\Gamma_{\underline{\mu}, P} } = \C \langle \alpha_{IJKH} \varphi^{I} \wedge \bar\varphi^{J} \wedge \psi^{K} \wedge \bar \psi^{H} \quad | \quad \alpha_{IJKH}|_{\Gamma_{\underline{\mu}, P} } = 1, |I| + |J| + |K| + |H| = p \rangle\]
where $\alpha_{IJKH} = \alpha_{i_{1}}\cdots\alpha_{h_{|H|}}$.
Then the following holds (see \cite[theorem 2.7]{Kasuya2}):
\begin{proposizione}\label{GeneratoriDeRhamGenNakamura}
    If $N = N_{\underline{\mu}, P}$ is defined as in \eqref{DefGenNakamura}, then $H^{*}_{d}(N) \simeq A^{*}_{\Gamma_{\underline{\mu}, P} }$.
\end{proposizione}
In particular, we observe that the de Rham cohomology groups do not depend on the choice of $k_{i} \in \mathbb{Z}$ in Theorem \ref{LatticeInGenNakamura} i.e., the generators and the \textit{Betti numbers} do not depend on the choice of the lattice.\\[5pt]
We now apply the techniques described in \cite{Kasuya} to compute generators for the Dolbeault cohomology of $N$. Firstly, we notice how our manifolds $N_{\underline{\mu}, P}$ satisfy all the conditions required in \cite[assumption 1.1]{Kasuya}. The characters $\alpha_{i}$ are given by 
\[\alpha_{2i+1} = e^{\lambda(w_{1} + \bar w_{2} + \cdots + \bar w_{2n})} \qquad \alpha_{2i+2} = e^{-\lambda(\bar w_{1} + w_{2} + \cdots + \bar w_{2n-1} + w_{2n})}\]
for $i = 0, \ldots, m-1$. Then, following \cite[lemma 2.2]{Kasuya}, we get the characters $\beta_{j}, \gamma_{l}$:
\[\beta_{j} = \begin{cases}
    e^{-\lambda(w_{2} - \bar w_{2} + \cdots + w_{2n} - \bar w_{2n})} \quad j \text{ odd } \\
    e^{\lambda(w_{1} - \bar w_{1} + \cdots + w_{2n-1} - \bar w_{2n-1})} \quad j \text{ even } \\
\end{cases} \qquad \gamma_{l} = \begin{cases}
    e^{-\lambda(w_{1} - \bar w_{1} + \cdots + w_{2n-1} - \bar w_{2n-1})} \quad l \text{ odd } \\
    e^{\lambda(w_{2} - \bar w_{2} + \cdots + w_{2n} - \bar w_{2n})} \quad l \text{ even } \\
\end{cases}\]
In particular, \textit{all the characters $\beta_{j}, \gamma_{l}$ are trivial on $\Gamma_{\underline{\mu}, P}$}. This implies that the sub-DGA of $A^{*,*}(N)$, $B^{*,*}$, is given by
\[B^{p,q} = \C\langle \beta_{J}\gamma_{L}\varphi^{I} \wedge \psi^{J} \wedge \bar\varphi^{K} \wedge \bar\psi^{L} \mid |I| + |J| = p, |K| + |L| = q \rangle\]
where $I,K \subset \{1, \ldots, 2n\}, J,L \subset \{1, \ldots, 2m\}$ are multiindices and the notations are the same as for the de Rham cohomology. A direct calculation from the structure equations \eqref{EqStrNakamuraNuove} leads to \[\delbar(\beta_{J}\gamma_{L}\varphi^{I} \wedge \psi^{J} \wedge \bar\varphi^{K} \wedge \bar\psi^{L}) = 0\] which means $\delbar|_{B^{p,q}} = 0$. This gives the following:
\begin{proposizione}\label{GeneratoriDolbeaultGenNakamura}
    For $N = N_{\underline{\mu}, P}$ as in \eqref{DefGenNakamura}, we have $H^{*}_{\delbar}(N) \simeq B^{*,*}$.
\end{proposizione}
The previous proposition directly implies that $N = N_{\underline{\mu}, P}$ \textit{cannot satisfy the $\del\delbar$-lemma}: indeed, from Proposition \ref{GeneratoriDeRhamGenNakamura} we get $b_{1}(N) = 4n$, while Proposition \ref{GeneratoriDolbeaultGenNakamura} gives $h^{1,0}(N) = h^{0,1}(N) = 2n+2m$. Thus, 
\[h^{1,0}(N) + h^{0,1}(N) = 4n + 4m > b_{1}(N)\] 
and the $\del\delbar$-lemma does not hold on $N$. 

\begin{osservazione}\label{NoDelDelBar}
    For $N = N_{\underline{\mu}, P}$ as defined in \eqref{DefGenNakamura}, we always have \[\sum_{p+q = k} h^{p,q}(N) > b_{k}(N).\]
\end{osservazione}

\subsubsection{Symplectic structures and Hard Lefschetz Condition} Let $N = N_{\underline{\mu}, P}$ as defined in \eqref{DefGenNakamura}. We remark that the eigenvalues $\{e^{\lambda_{1}}, \ldots, e^{\lambda_{2m}}\}$ of the matrix $D$ satisfy $\lambda_{2i+1} + \lambda_{2i+2} = 0$ for $i = 0, \ldots, m-1$ (up to possible rearrangements of them). Then, as a generalization of \eqref{SimpletticheHLCGenNakamura}, we prove the following:

\begin{teorema}\label{SimpletticheHLCNuove}
    $N = N_{\underline{\mu}, P} $ carries invariant symplectic structures of the form
    \[  \omega = \dfrac{i}{2} \sum_{k = 1}^{2n}C_{k}\varphi^{k} \wedge \bar\varphi^{k} + \dfrac{1}{2}\sum_{i = 0}^{m-1}\left(B_{i}\psi^{2i+1} \wedge \bar\psi^{2i+2} + \bar B_{i}\bar\psi^{2i+1} \wedge \psi^{2i+2} \right)\]
    where $C_{k} \in \mathbb{R}\setminus \{0\} \quad \forall k$ and $B_{i} \in \C\setminus \{0\} \quad \forall i$. Moreover, all these structures satisfy the Hard Lefschetz Condition.
\end{teorema}

\begin{proof}
    We just need to prove HLC. Using the same notation as in theorem \ref{SimpletticheHLCGenNakamura}, the powers of $\omega$ are given by
    \[\omega^{k} = k! \sum_{\substack{2|K| + |H| + |R| = 2k \\ c_{HR} = 0}} A_{KHR}\varphi^{K} \wedge \bar\varphi^{K} \wedge \psi^{H} \wedge \bar\psi^{R}\]
    with coefficients $\tilde{A}_{KHR} \ne 0$. Now we consider
    \[L^{k}: H^{2n+2m-k}_{d}(N) \rightarrow H^{2n+2m+k}_{d}(N)\]
    and compute the image of the generators. From proposition \ref{GeneratoriDeRhamGenNakamura} and the previous remarks on the characters $\alpha_{i}$, we get three possible types of generators for $H^{2n+2m-k}_{d}(N)$:\\[5pt]
    \begin{itemize}
        \item $\varphi^{I} \wedge \bar\varphi^{I'} \wedge \psi^{J} \wedge \bar\psi^{J'}$ such that $c_{JJ'} = 0$ and $|I| + |I'| + |J| + |J'| = 2n+2m-k$;\\[5pt]
        \item $F\varphi^{I} \wedge \bar\varphi^{I'} \wedge \psi^{J} \wedge \bar\psi^{J'}$ such that $c_{JJ'} = 0$ and $|I| + |I'| + |J| + |J'| = 2n+2m-k$;\\[5pt]
        \item $\bar F\varphi^{I} \wedge \bar\varphi^{I'} \wedge \psi^{J} \wedge \bar\psi^{J'}$ such that $c_{JJ'} = 0$ and $|I| + |I'| + |J| + |J'| = 2n+2m-k$;\\[5pt]
    \end{itemize}
    where $F := e^{\lambda(w_{1} - \bar w_{1} - w_{2} + \bar w_{2} + \ldots -w_{2n} + \bar w_{2n})}$. Then we get:\\
    \[L^{k}\varphi^{I} \wedge \bar\varphi^{I'} \wedge \psi^{J} \wedge \bar\psi^{J'} = k!\sum_{\substack{2|K| + |H| + |R| = 2k \\ c_{HR} = 0 \\ K \cap I = \emptyset, K \cap I' = \emptyset \\ H \cap J = \emptyset, R \cap J' = \emptyset}} A'_{KHR}\varphi^{K \cup I} \wedge \bar\varphi^{K \cup I'} \wedge \psi^{H \cup J} \wedge \bar\psi^{R \cup J'}\]
    \[L^{k}F\varphi^{I} \wedge \bar\varphi^{I'} \wedge \psi^{J} \wedge \bar\psi^{J'} = k!\sum_{\substack{2|K| + |H| + |R| = 2k \\ c_{HR} = 0 \\ K \cap I = \emptyset, K \cap I' = \emptyset \\ H \cap J = \emptyset, R \cap J' = \emptyset}} A''_{KHR}F\varphi^{K \cup I} \wedge \bar\varphi^{K \cup I'} \wedge \psi^{H \cup J} \wedge \bar\psi^{R \cup J'}\]
    \[L^{k}\bar F\varphi^{I} \wedge \bar\varphi^{I'} \wedge \psi^{J} \wedge \bar\psi^{J'} = k!\sum_{\substack{2|K| + |H| + |R| = 2k \\ c_{HR} = 0 \\ K \cap I = \emptyset, K \cap I' = \emptyset \\ H \cap J = \emptyset, R \cap J' = \emptyset}} A'''_{KHR} \bar F\varphi^{K \cup I} \wedge \bar\varphi^{K \cup I'} \wedge \psi^{H \cup J} \wedge \bar\psi^{R \cup J'}\]\\
    which are independent generators for $H^{2n+2m+k}_{d}(N)$ as a consequence of the non-degeneracy of $\omega$ and the couples of indices appearing in the definition of $\omega$.
\end{proof}

\subsubsection{Non-K\"ahlerianity} We conclude by showing that the manifolds $N = N_{\underline{\mu}, P}$, with the complex structure defined in \eqref{EqStrNakamuraNuove}, are not K\"ahler. We remark how non-K\"ahlerianity follows from the characterization of compact solvmanifolds with K\"ahler structures provided by K. Hasegawa in \cite{Hasegawa}; in the following we give a simple proof of this for our manifold $N$. Before, we recall some well-known definitions which will be applied in the proof. We set $\sigma_p := i^{p^2}2^{-p}$.
\begin{definizione}
    Let $(M,J)$ be a complex manifold of complex dimension $n$. A $p$-co-vector $\beta \in \Lambda^{p,0}(T_xM \otimes \C)$ is said to be \textit{simple} if $\beta = \beta^1 \wedge ... \wedge \beta^p$, where $\beta^i \in \Lambda^{1,0}(T_xM \otimes \C)$ for all $i=1, ..., p$.
\end{definizione}
\begin{definizione}
    A real $p$-co-vector $\psi^p \in \Lambda^{p,p}(T_xM)$ is called \textit{transverse} if 
    \[\sigma_{n-p}\psi^p \wedge \beta \wedge \bar\beta\]
    is strictly positive for all $\beta \in \Lambda^{n-p,0}(T_xM \otimes \C)$ simple, where \textit{strictly positive} means that $\sigma_{n-p}\psi^p \wedge \beta \wedge \bar\beta = c\text{Vol}$, $\text{Vol} = \sigma_n\psi^1 \wedge ... \wedge \psi^n \wedge \bar\psi^1 \wedge ... \wedge \bar\psi^n $ with $\{\psi^1, ..., \psi^n\}$ a basis for $\Lambda^{n,0}(T_xM \otimes \C)$. 
\end{definizione}
\begin{definizione}\label{pKahler}
    Let $(M,J)$ be a complex manifold of complex dimension $n$, and let $1 \leq p \leq n$. Given be a real $(p,p)$-form $\psi$, we say that $\psi$ is a $p$-{\em K\"ahler} structure if $d\psi = 0$ and $\psi_{x} \in \Lambda^{p,p}_{\mathbb{R}}(T_{x}^{\ast}M)$ is transverse $\forall x \in M$.
\end{definizione}
\begin{osservazione}\label{KahlerE1kahler}
Recall that a complex manifold $(M, J)$ is K\"ahler if and only if it is $1$-K\"ahler, as proved in \cite[proposition 1.15]{AA}.
\end{osservazione}
We now prove the following proposition.
\begin{proposizione}\label{NNotKahler}
    Let $N = N_{\underline{\mu}, P}$ be a manifold as in \eqref{DefGenNakamura} of complex dimension $2n+2m$. Then $N$ is not K\"ahler.
\end{proposizione}
\begin{proof}
    Suppose there exists a K\"ahlerian metric $\omega$ on $N$. Take the non-trivial $(2n+2m-1, 2n+2m-1)$-form
    \[\phi =\varphi^{1} \wedge \bar\varphi^{1} \wedge \cdots \wedge \varphi^{2n} \wedge \bar\varphi^{2n} \wedge\psi^{1}\wedge \bar\psi^{1}\cdots \wedge \psi^{2m-1} \wedge \bar\psi^{2m-1} \]
    $\phi$ is $d$-exact, indeed:
    \[d\bigg[\frac{1}{(2m-1)\lambda}\varphi^{1} \wedge \bar\varphi^{1} \wedge \cdots \wedge \varphi^{2n}\wedge\psi^{1}\wedge \bar\psi^{1}\cdots \wedge \psi^{2m-1} \wedge \bar\psi^{2m-1}\bigg] = \]
    \[= \frac{1}{(2m-1)\lambda}\varphi^{1} \wedge \bar\varphi^{1} \wedge \cdots \wedge \varphi^{2n} \wedge \bar\varphi^{2n}\wedge [(2m-1)\lambda\psi^{1}\wedge \bar\psi^{1}\cdots \wedge \psi^{2m-1} \wedge \bar\psi^{2m-1}] = \phi.\]
    Set 
    $$
    \theta = \frac{1}{(2m-1)\lambda}\varphi^{1} \wedge \bar\varphi^{1} \wedge \cdots \wedge \varphi^{2n}\wedge\psi^{1}\wedge \bar\psi^{1}\cdots \wedge \psi^{2m-1} \wedge \bar\psi^{2m-1}
    $$
    so that $\phi = d\theta$. Up to a multiplicative constant, $\phi = \beta \wedge \bar\beta$ with $\beta = \varphi^{1} \wedge ... \wedge \varphi^{2n} \wedge \psi^1 \wedge ... \wedge \psi^{2m-1}$ simple.
    Then by Stokes's theorem and remark \ref{KahlerE1kahler} we get
    \[0 < \int_{N} \sigma_{2n+2m-1}\,\omega \wedge \beta \wedge \bar\beta = \int_{N} \sigma_{2n+2m-1} \,\omega \wedge \phi = \int_{N} d(\omega \wedge \sigma_{2n+2m-1}\,\theta) = 0\]
    which is absurd.
\end{proof}

\section{Kodaira dimension of almost-complex structures}
We now compute the \textit{Kodaira dimension} of some almost-complex structures on the generalized Nakamura manifold and $N$ compatible with the symplectic structures we found. Through these computations we provide almost-complex structures compatible with symplectic forms which turn out to have Kodaira dimension $0$ and $-\infty$. In particular, the case $ -\infty $ is obtained following the techniques in \cite{CNT}.

\subsection{Kodaira dimension of almost-complex structures on $\Gamma_{P, \tau} \backslash (\C \ltimes_{\rho} \C^{4})$} 
Let $$ N_{5} =  \Gamma_{P, \tau} \backslash (\C \ltimes_{\rho} \C^{4})$$ be the generalized Nakamura manifold as  in \eqref{DefinizioneGenNakamura}. We defined a global co-frame of $ (1,0) $-forms, $ \{\varphi^{0}, \varphi^{1}, \varphi^{2}, \varphi^{3}, \varphi^{4}\} $, with structure equations \eqref{StrEqn}. We take on $\C \ltimes_{\rho} \C^{4}  $ coordinates $ (z,z_{1}, z_{2}, z_{3}, z_{4}) $ and set 
 \[ z = x + iy \qquad z_{j} = x_{j} +iy_{j}  \quad \text{ for }  j = 1, \ldots, 4. \]
In this coordinate system, the co-frame is written as
\begin{equation*}
\left\{
\begin{array}{lll}
	\varphi^{0} &=& dx + i dy \\
	\varphi^{1} &=& e^{-\lambda_{1}x}(dx_{1} + idy_{1}) \\ \varphi^{2}&=& e^{\lambda_{1}x}(dx_{2} + idy_{2}) \\
	\varphi^{3} &=& e^{-\lambda_{2}x}(dx_{3} + idy_{3}) \\	\varphi^{4} &=& e^{\lambda_{2}x}(dx_{4} + idy_{4}). \\
\end{array}
\right.
\end{equation*}
Now we set a new real co-frame on $ N_{5} $:
\begin{equation*}
	\begin{split}
		& E^{0} = dx \qquad\qquad\,\,\,\,\,F^{0} = dy \\
		& E^{1} = e^{-\lambda_{1}x}dx_{1} \qquad F^{1} = e^{\lambda_{1}x}dx_{2}\\
		& E^{2} = e^{-\lambda_{1}x}dy_{1}  \qquad F^{2} = e^{\lambda_{1}x}dy_{2} \\
		& E^{3} = e^{-\lambda_{2}x}dx_{3} \qquad F^{3} = e^{\lambda_{2}x}dx_{4} \\
		& E^{4} = e^{-\lambda_{2}x}dy_{3} \qquad F^{4} = e^{\lambda_{2}x}dy_{4} \\
	\end{split}
\end{equation*}
and let $ \{E_{i}, F_{i}\}_{i = 0, \ldots, 4} $ be its dual frame. We define an almost-complex structure $ J $ by setting
\[ JE_{i} = F_{i} \qquad JF_{i} = - E_{i} \]
for $ i = 0, \ldots, 4 $. We can then take a symplectic structure on $ N_{5} $, $\tilde{\omega}$, such that the just defined $ J $ is \textit{$\tilde{\omega}$-compatible}, setting
\[ \tilde{\omega} = \sum_{j = 0}^{4} E^{j} \wedge F^{j}. \]
The real 2-form $ \tilde{\omega} $ is $ d $-closed and non-degenerate, giving a well-defined symplectic structure on $ N_{5} $. We remark that $  \tilde{\omega}  $ corresponds, in terms of the co-frame $ \{\varphi^{0}, \varphi^{1}, \varphi^{2}, \varphi^{3}, \varphi^{4}\} $, to
\[ \tilde{\omega} = \dfrac{i}{2}\varphi^{0\bar 0} + \dfrac{1}{2}(\varphi^{1\bar 2} + \varphi^{\bar 1 2} + \varphi^{3 \bar 4} + \varphi^{3 \bar 4}) \]
where $ \varphi^{ij} = \varphi^{i} \wedge \varphi^{j} $. Note that this agrees with theorem \ref{SimpletticheHLCGenNakamura}. We will not work directly with $ J $ in the following computations, but with a deformation of it.\\
Let $ J_{t} $ be the almost-complex structure defined by taking the following global co-frame of $ (1,0) $-forms on $ N_{5} $:
\begin{equation*}
\left\{
	\begin{split}
		& \phi_{t}^{0} = E^{0} + iF^{0} \\
		& \phi_{t}^{1} = E^{1} + iF^{1} \\
		& \phi_{t}^{2} = E^{2} + iF^{2} \\
		& \phi_{t}^{3} = E^{3} + iF^{3} \\
		& \phi_{t}^{4} = E^{4} -i\Big[\alpha(t)E^{4} +\beta(t)F^{4}\Big]
	\end{split}
    \right.
\end{equation*}
with $ t = (t_{1}, t_{2}) \in \mathbb{R}^{2} $, $ |t| $ sufficiently small, and $ \alpha(t), \beta(t) $ given by
\[ \alpha(t) = \dfrac{2t_{2}}{t_{1}^{2}+t_{2}^{2} -1} \qquad \beta(t) = \dfrac{(1-t_{1})^{2} + t_{2}^{2}}{t_{1}^{2}+t_{2}^{2} -1}.\]
Precisely, the structure $ J_{t} $ can be described in terms of the matrix
\[ \begin{pmatrix}
	0 & -1 & & & & & & & & \\
	1 & 0 & & & & & & & & \\
	& & 0 & -1 & & & & & & \\
	& & 1 & 0 & & & & & & \\
	& & & &0 & -1 & & & &  \\
	& & & &1 & 0 & & & & \\
	& & & & & & 0 & -1 & & \\
	& & & & & & 1 & 0 & &  \\
	& & & & & & & &  \alpha(t) & \beta(t) \\
	& & & & & & & &  \gamma(t) & -\alpha(t) \\
\end{pmatrix} \]
with $\gamma(t)$ such that $ -\alpha(t)^{2} -\gamma(t)\beta(t) = 1 $. A straightforward calculation from the differentials of $ E^{i}, F^{i} $ leads to the following structure equations for the co-frame $ \{\phi_{t}^{0}, \phi_{t}^{1} , \phi_{t}^{2},\phi_{t}^{3} , \phi_{t}^{4}\} $:
\begin{equation}\label{EqStrutturaKodairaGenNakamura}
\left\{
	\begin{split}
    	& d\phi_{t}^{0} = 0 \\
		& d\phi_{t}^{1} = -\dfrac{\lambda_{1}}{2}\phi_{t}^{0\bar 1} -\dfrac{\lambda_{1}}{2}\phi_{t}^{\bar 0 \bar 1} \\
		&d\phi_{t}^{2} = -\dfrac{\lambda_{1}}{2}\phi_{t}^{0\bar 2} -\dfrac{\lambda_{1}}{2}\phi_{t}^{\bar 0 \bar2} \\
		&d\phi_{t}^{3} = -\dfrac{\lambda_{2}}{2}\phi_{t}^{0\bar 3} -\dfrac{\lambda_{2}}{2}\phi_{t}^{\bar 0 \bar3} \\
		&d\phi_{t}^{4} = -\dfrac{\lambda_{2}}{2}(\phi_{t}^{0} + \phi_{t}^{\bar0}) \wedge [(1-i\alpha(t))\phi_{t}^{\bar4} - i\alpha(t)\phi_{t}^{4}]. \\
	\end{split}
    \right.
\end{equation}
The complex frame dual to $ \{\phi_{t}^{0}, \phi_{t}^{1} , \phi_{t}^{2},\phi_{t}^{3} , \phi_{t}^{4}\}$ is determined by the vector fields $ \{V_{0t}, V_{1t}, V_{2t}, V_{3t}, V_{4t}\} $, defined as
\begin{equation}\label{RifDualeKodairaGenNakamura}
\left\{
	\begin{split}
		&V_{0t} = \dfrac{1}{2}(E_{0} - iF_{0}) \\
		&V_{1t} = \dfrac{1}{2}(E_{1} - iF_{1}) \\
		&V_{2t} = \dfrac{1}{2}(E_{2} - iF_{2}) \\
		&V_{3t} = \dfrac{1}{2}(E_{3} - iF_{3}) \\
		&V_{4t} = \dfrac{1}{2}\bigg(E_{4} - \dfrac{\alpha(t)}{\beta(t)}F_{4} + \dfrac{i}{\beta(t)}F_{4}\bigg). \\
	\end{split}
    \right.
\end{equation}
We are now able to compute the almost-complex Kodaira dimension of $ (N_{5}, J) $. Precisely, we prove the following result:
\begin{proposizione}
	Let $ (N_{5}, J, \tilde{\omega}) $ as previously defined. Then  
	\[ \kappa^{J}(N_{5}) = \begin{cases}
		-\infty & \text{ if } \alpha(t) \notin \dfrac{2\pi}{k\lambda_{2}}\mathbb{Z} \quad \forall k \ge 1 \\
		0 & \text{ otherwise } \\
	\end{cases}. \] 
\end{proposizione}
\begin{proof}
	Let $ \Omega_{t} = \phi_{t}^{01234} \in \Lambda^{5,0}(N_{5})$. From \eqref{EqStrutturaKodairaGenNakamura} we get
	\[ \delbar\Omega_{t} = \dfrac{i\lambda_{2}\alpha(t)}{2} \phi_{t}^{\bar 0 01234}\]
	which is non-zero if we take $ \alpha(t) \ne 0 $. Consider now a smooth section $ \sigma = F\Omega_{t} $ of $\Lambda^{5,0}(N_{5})  $. Then $ F\Omega_{t} \in H^{0}(N_{5}, \Lambda^{5,0}(N_{5})) $ if and only if
	\[ 0 = \delbar(F\Omega_{t}) = \delbar F \wedge  \phi_{t}^{01234}  +   \dfrac{i\lambda_{2}\alpha(t)}{2} F\phi_{t}^{\bar 0 01234} \iff \delbar F +   \dfrac{i\lambda_{2}\alpha(t)}{2} F\phi_{t}^{\bar 0} = 0.\]
	The equation translates into the system of differential equations on $ N_{5} $ given by
	\[ \begin{cases}
		\bar V_{0t}F +   \dfrac{i\lambda_{2}\alpha(t)}{2} F = 0 \\
		\bar V_{1t}F = \bar V_{2t}F = \bar V_{3t}F = \bar V_{4t}F = 0. \\
	\end{cases} \]
	The last four equations and their conjugates give us
	\[ (V_{1t}\bar V_{1t} + V_{2t}\bar V_{2t} + V_{3t}\bar V_{3t} + V_{4t}\bar V_{4t})F = 0. \]
	Setting $ \mathcal{L} =  V_{1t}\bar V_{1t} + V_{2t}\bar V_{2t} + V_{3t}\bar V_{3t} + V_{4t}\bar V_{4t}$, a direct computation fron \eqref{RifDualeKodairaGenNakamura} shows that
	\begin{equation*}
			\begin{split}
			\mathcal{L} & = e^{2\lambda_{1}x}\dfrac{\del^{2}}{\del x_{1}^{2}} + e^{-2\lambda_{1}x}\dfrac{\del^{2}}{\del x_{2}^{2}} + e^{2\lambda_{1}x}\dfrac{\del^{2}}{\del y_{1}^{2}} + e^{-2\lambda_{1}x}\dfrac{\del^{2}}{\del y_{2}^{2}} + e^{2\lambda_{2}x}\dfrac{\del^{2}}{\del x_{3}^{2}} + e^{-2\lambda_{2}x}\dfrac{\del^{2}}{\del x_{4}^{2}} + e^{2\lambda_{2}x}\dfrac{\del^{2}}{\del y_{3}^{2}} + \\
			& -2\dfrac{\alpha(t)}{\beta(t)}\dfrac{\del^{2}}{\del y_{3}\del y_{4}} + e^{-2\lambda_{2}x}\bigg(1 + \dfrac{\alpha^{2}(t)}{\beta^{2}(t)}\bigg)\dfrac{\del^{2}}{\del x_{1}^{2}}
		\end{split}
	\end{equation*}
	and $\mathcal{L}$ turns out to be elliptic. Writing $ F = u + iv $ the equation $ \mathcal{L}F = 0 $ is equivalent to $ \mathcal{L}u = \mathcal{L}v = 0 $. Therefore, the compactness of $ N_{5} $ implies that $ u = u(x,y)$ and $ v = v(x,y) $. We now consider the first equation of the system, which further translates into a differential system in $ u $ and $ v $:
	\[ \begin{cases}
		u_{x} - v_{y} - \lambda_{2}\alpha(t)v = 0 \\
		u_{y} + v_{x} +\lambda_{2}\alpha(t)u = 0. \\
	\end{cases} \]
	Differentiating the first one with respect to $ x $, the second one with respect to $ y $ and adding the two, we get 
	\[ u_{xx} + u_{yy} + \lambda_{2}\alpha(t)u_{y} - \lambda_{2}\alpha(t)v_{x} = 0  \]
	which becomes, re-writing $ v_{x} $ from the second equation
	\begin{equation}\label{EqPerUKodairaGenNakamura}
		u_{xx} + u_{yy} + 2\lambda_{2}\alpha(t)u_{y} + \lambda_{2}^{2}\alpha^{2}(t)u = 0.
	\end{equation} 
	We have now obtained an equation regarding just $ u $, and we need to look at periodic solutions, in the form of Fourier series
	\[ u(x,y) = \sum_{m,n \in \mathbb{Z}}^{}u_{mn}e^{2\pi i(mx + \frac{n}{\tau}y)} \]
	where $\tau$ is the parameter appearing in the definition of the lattices in $ \C \ltimes_{\rho} \C^{4} $. Substituting then in \eqref{EqPerUKodairaGenNakamura}, we get for each couple $ m,n \in \mathbb{Z} $ the equation
	\[ \bigg[-4\pi^{2}\bigg(m^{2} + \dfrac{n^{2}}{\tau^{2}}\bigg) + 4\pi i\dfrac{n}{\tau}\lambda_{2}\alpha(t) +  \lambda_{2}^{2}\alpha^{2}(t)\bigg]u_{mn} = 0. \]
	Suppose now there exists a couple $ (\bar m, \bar n) \ne (0,0) $ such that $ u_{\bar m \bar n} \ne 0 $. Then the previous equation gives
	\[ \begin{cases}
		-4\pi^{2}\bigg(\bar m^{2} + \dfrac{\bar n^{2}}{\tau^{2}}\bigg) + \lambda_{2}^{2}\alpha^{2}(t) = 0 \\[8pt]
		4\pi \dfrac{\bar n}{\tau}\lambda_{2}\alpha(t) = 0. \\
	\end{cases} \]
	Since we have already taken $ \alpha(t) \ne 0 $, the second equation directly gives $ \bar n  = 0$ and the first one becomes
	\[ -4\pi^{2}\bar m^{2} + \lambda_{2}^{2}\alpha^{2}(t) = 0 \]
	which has a solution $ \bar m \in \mathbb{Z} $ if and only if $ \alpha(t) \in  \dfrac{2\pi}{\lambda_{2}}\mathbb{Z} $. Hence, if we take $ \alpha(t) \notin  \dfrac{2\pi}{\lambda_{2}}\mathbb{Z} $, it follows that $ u_{mn} = 0 $ for all $ m,n \ne 0 $. Therefore, $ u = u_{00} $ is constant. The same argument for $ v $ also gives us $ v = v_{00} $. Putting everything together, if $  \alpha(t) \notin  \dfrac{2\pi}{\lambda_{2}}\mathbb{Z}  $ the equation $ \mathcal{L}F = 0 $ is satisfied if and only if $ F $ is constant. Now, if $ F \ne 0 $, the initial equation $ 0 = \delbar(F\Omega_{t}) $ reduces to
	\[  \dfrac{i\lambda_{2}\alpha(t)}{2} F\phi_{t}^{\bar 0 01234} = 0  \]
	which is absurd. We have then proved that
	\[  \alpha(t) \notin  \dfrac{2\pi}{\lambda_{2}}\mathbb{Z} \implies F = 0. \]
	Hence, there are no trivial sections of $ \Lambda^{5,0}(N_{5}) $. Summarizing, we got
	\[ P_{1,J} = \begin{cases}
		0 \text{ if }  \alpha(t) \notin  \dfrac{2\pi}{\lambda_{2}}\mathbb{Z}  \\
		1 \text{ otherwise }. \\
	\end{cases} \]
	The same argument can be generalized to $( \Lambda^{5,0}(N_{5}) )^{\otimes k}$ for $ k > 1 $. In fact, taking a smooth section of $ ( \Lambda^{5,0}(N_{5}) )^{\otimes k} $
	\[ \sigma = F\Omega_{t} \otimes \cdots \otimes \Omega_{t},\]
	the same steps lead to the equation
	\[ \delbar F + \dfrac{ik\lambda_{2}\alpha(t)}{2}F = 0. \]
	Thus, proceding in the same way as before, we get
	\[ P_{k,J} = \begin{cases}
		0 \text{ if }  \alpha(t) \notin  \dfrac{2\pi}{k\lambda_{2}}\mathbb{Z} \\
		1 \text{ otherwise }. \\
	\end{cases}  \]
	This proves the proposition.
\end{proof}

\begin{osservazione}
	\begin{enumerate}
		\item The condition $ \alpha(t) \notin \dfrac{2\pi}{k\lambda_{2}}\mathbb{Z} \quad \forall k \ge 1 $ is non-empty (take indeed $ t = (1,1) $).
		\item We point out that the validity of this method holds for every generalized Nakamura manifold of odd complex dimension; we do not carry out the explicit computations as the indices and the number of equations rapidly increase. However, the method allows the production of many examples of symplectic manifolds endowed with almost-complex structures compatible with the symplectic form with Kodaira dimension $ -\infty $.
	\end{enumerate}
\end{osservazione}

\subsection{Kodaira dimension of almost-complex structures on $N = N_{\underline{\mu}, P}$}
We consider an almost-complex structure on $N$ which is $\omega$-compatible, with $\omega$ the symplectic structure given by
\[\omega = \dfrac{i}{2} \sum_{k = 1}^{2n}\varphi^{k} \wedge \bar\varphi^{k} + \dfrac{1}{2}\sum_{i = 0}^{m-1}\left(\psi^{2i+1} \wedge \bar\psi^{2i+2} + \bar\psi^{2i+1} \wedge \psi^{2i+2} \right).\]
From the structure equations \eqref{EqStrNakamuraNuove}, $N$ is a complex manifold  but the complex structure is not $\omega$-compatible. We remark that $N$ cannot admit almost-complex structures which are both $\omega$-compatible and integrable, as it is not K\"ahler. However, we can define a  $\omega$-compatible almost-complex structure on $N$. Consider the 1-forms
\begin{equation}\label{ACmplxStrOnNewManifolds}
\left\{
\begin{array}{llll}
\alpha^{i} &=& \varphi^{i} \qquad & \text{for } i = 1,\ldots, 2n\\ [3pt] \beta^{2j+1} &=& \frac{\sqrt{2}}{2}\bigl(\psi^{2j+1} + \sqrt{-1}\psi^{2j+2}\bigr) {} &\\[3pt]
\beta^{2j+2} &=& \frac{\sqrt{2}}{2}\bigl(\bar\psi^{2j+1} + \sqrt{-1}\bar\psi^{2j+2}\bigr) \qquad &\text{for } j = 0, ..., m-1.
\end{array}    
\right.
\end{equation}
We set $$\{\alpha^{1}, \ldots, \alpha^{2n}, \beta^{1}, \ldots, \beta^{2m}\}
$$
as a co-frame of $(1,0)$-forms on $N$. Keeping the notation $$\eta := \varphi^{1} + \bar\varphi^{2} + \cdots + \bar \varphi^{2n} = \alpha^{1} + \bar\alpha^{2} + \cdots + \bar \alpha^{2n},$$ from \eqref{EqStrNakamuraNuove} we get the following structure equations:
$$
\left\{
\begin{array}{llll}
    d\alpha^{i} &=& 0 \qquad &\text{for } i = 1, \ldots, 2n \\[3pt]
    d\beta^{2j+1} &=& \frac{\lambda}{2}\bigl[(\bar\eta - \eta) \wedge \beta^{2j+1} - (\eta + \bar\eta) \wedge \bar\beta^{2j+2}\bigr] \qquad & \text{for } j = 0, \ldots, m-1 \\[3pt] 
    d\beta^{2j+2} &=& \frac{\lambda}{2}\bigl[-(\bar\eta + \eta) \wedge \bar\beta^{2j+1} + (\eta - \bar\eta) \wedge \beta^{2j+2}\bigr] \qquad & \text{for } j = 0, \ldots, m-1.\\[3pt]
\end{array}
\right.
$$
The obtained equations define an almost-complex structure on $N$ which is not integrable. We will denote it by $\mathcal{J}$. It turns out that $\mathcal{J}$ is $\omega$-compatible as in this new co-frame we get
\[\omega = \frac{\sqrt{-1}}{2}\bigl(\sum_{i = 1}^{2n} \alpha^{i} \wedge \bar\alpha^{i} + \sum_{j = 1}^{2m} \beta^{j} \wedge \bar\beta^{j}\bigr).\]
We now compute the \textit{Kodaira dimension} of $(N, \mathcal{J})$. From the structure equations of the co-frame \eqref{ACmplxStrOnNewManifolds}, we get:
$$
\left\{
\begin{array}{lll}
  \delbar\alpha^{i} &= 0 \qquad &\text{for } i = 1, \ldots, 2n \\[3pt]
  \delbar\beta^{2j+1} &= \frac{\lambda}{2}\bigl[\bigl(-\sum_{\substack{i = 1 \\ i \text{ }  even}}^{2n}\bar\alpha^{i} + \sum_{\substack{i = 1 \\ i \text{ } odd}}^{2n}\bar\alpha^{i}\bigr) \wedge \beta^{2j+1} - \bigl(\sum_{i = 1}^{2n} \alpha^{i}\bigr) \wedge \bar\beta^{2j+2}\bigr] \qquad & \text{for } j = 0, ..., m-1 \\[3pt]
  \delbar\beta^{2j+2} &= \frac{\lambda}{2}\bigl[-\bigl(\sum_{i = 1}^{2n} \alpha^{i}\bigr) \wedge \bar\beta^{2j+1} + \bigl(\sum_{\substack{i = 1 \\ i \text{ }  even}}^{2n}\bar\alpha^{i} - \sum_{\substack{i = 1 \\ i \text{ } odd}}^{2n}\bar\alpha^{i}\bigr) \wedge \beta^{2j+2}\bigr] \qquad &\text{for } j = 0, \ldots, m-1. \\[3pt]
\end{array}
\right.
$$
We can then prove the following:
\begin{proposizione}\label{KodairaNuove}
    Let $(N, \mathcal{J})$ be the almost-complex manifold with a complex $(1,0)$-coframe given by  \eqref{ACmplxStrOnNewManifolds}. Then $\kappa^{\mathcal{J}}(N) = 0$.
\end{proposizione}

\begin{proof}
    Fix $\Omega = \alpha^{1} \wedge \cdots \wedge \alpha^{2n} \wedge \beta^{1} \wedge \cdots \wedge \beta^{2m} \in \Lambda^{2n+2m,0}(N)$. From now on we make use of the notation $\alpha^{1\ldots 2n} = \alpha^{1} \wedge \cdots \wedge \alpha^{2n}$. From the previous calculations we get:
    \[\delbar\Omega = \alpha^{1\ldots 2n} \wedge \bigl[\sum_{j=1}^{2m}(-1)^{j-1}\beta^{1} \wedge \cdots \wedge \delbar\beta^{j} \wedge \cdots \wedge \beta^{2m}\bigr] = \]
    \[= \frac{\lambda}{2}\alpha^{1\cdots 2n}\bigl[\sum_{\substack{j = 1 \\ j \text{ } odd}}^{2m}\bigl(-\sum_{\substack{i = 1 \\ i \text{ }  even}}^{2n}\bar\alpha^{i} + \sum_{\substack{i = 1 \\ i \text{ } odd}}^{2n}\bar\alpha^{i}\bigr) \wedge \beta^{1\ldots2m} + \sum_{\substack{j = 1 \\ j \text{ } even}}^{2m}\bigl(\sum_{\substack{i = 1 \\ i \text{ }  even}}^{2n}\bar\alpha^{i} - \sum_{\substack{i = 1 \\ i \text{ } odd}}^{2n}\bar\alpha^{i}\bigr) \wedge \beta^{1\ldots 2m}\bigr] = \]
    \[= \frac{\lambda m}{2}\alpha^{1\ldots 2n}\bigl(-\sum_{\substack{i = 1 \\ i \text{ }  even}}^{2n}\bar\alpha^{i} + \sum_{\substack{i = 1 \\ i \text{ }  odd}}^{2n}\bar\alpha^{i} + \sum_{\substack{i = 1 \\ i \text{ }  even}}^{2n}\bar\alpha^{i} - \sum_{\substack{i = 1 \\ i \text{ }  odd}}^{2n}\bar\alpha^{i}\bigr)\wedge \beta^{1\ldots 2m} = 0.\]

    Hence, if $F\Omega \in H^{0}(N, \Lambda^{2n+2m,0}(N))$, $F \in C^{\infty}(N)$, we get 
    \[0 = \delbar(F\Omega) = \delbar F \wedge \Omega\]
    which implies $\delbar F = 0$. Thus, $F$ is constant, since $N$ is compact and we get $P_{1, \mathcal{J}} = 1$.
    Now let $s=F(\Omega \otimes \cdots \otimes \Omega)$ be a nowhere vanishing smooth section of $(\Lambda^{2n+2m,0}(N))^{\otimes r}$. 
    From the previous calculation we get $\delbar s=0$ if and only if $F$ is constant.
    We then get 
    \[P_{m, \mathcal{J}} = 1 \quad \forall r \ge 1\]
    This means that $(N, \mathcal{J})$ is holomorphically trivial and we get $\kappa^{\mathcal{J}}(N) = 0$.
\end{proof}

\section{Examples} We provide some explicit families of symplectic Nakamura manifolds. The examples will show different aspects of our constructions, in particular:
\begin{itemize}
    \item the assumptions on the eigenvalues are non-empty and non-restrictive, both in the case of quotients of $\C \ltimes_{\rho} \C^{m}$ and $\C^{2n} \ltimes_{\rho} \C^{2m}$;
    \item the splitting of a non-trivially splittable lattice in $\C \ltimes_{\rho} \C^{2n+1}$.
\end{itemize}

\begin{esempio}\label{ex:1}
    Consider $M \in \SL(2,\mathbb{Z})$ with $\hbox{tr}(M)\geq 2$. Then the eigenvalues of $M$ are given by 
    $$\mu = \dfrac{n \pm \sqrt{n^{2}-4}}{2}$$ 
    for some $n \in \mathbb{N}, n \ge 2$. If $n = 2$ we obtain a complex torus, so we set $n > 2$. Then we can take matrices $M$ in the form
    \[M = \begin{pmatrix}
        1 & 1 \\
        n-2 & n-1 \\
    \end{pmatrix}\]
    and the corresponding $D$ is such that $\lambda = \log(\frac{n + \sqrt{n^{2}-4}}{2})$. Given these matrices, we can build generalized Nakamura manifolds of dimension 3: $N_{3} = \Gamma_{P, \tau} \backslash (\mathbb{C} \ltimes_{\rho} \mathbb{C}^{2})$. With the same notation of Theorem \ref{SimpletticheHLCGenNakamura} a symplectic form on $N_{3}$ is given by
    \[\omega = \frac{i}{2}C\varphi^{0} \wedge \bar\varphi^{0} + \frac{1}{2}\bigl(A\varphi^{1} \wedge \varphi^{2} + B\varphi^{1} \wedge \bar\varphi^{2} + \bar B\bar\varphi^{1} \wedge \varphi^{2} + \bar A\bar\varphi^{1} \wedge \bar\varphi^{2}\bigr)\]
    with $C \in \mathbb{R} \setminus \{0\}$ and $A,B \in \C$, not both zero. The generators and Betti numbers of $N_{3}$ are summarized in table \ref{table:1}. We write $\varphi^{ij} = \varphi^{i} \wedge \varphi^{j}$.
\end{esempio}

\begin{esempio}\label{ex:2}
    Consider $\C \ltimes_{\rho} \C^{3}$, with $\rho$ given by a matrix $M \in \SL(3, \mathbb{Z})$ such that $D = \hbox{diag}(e^{\lambda}, e^{-\lambda}, 1)$. We aim to show that we can take non-trivial matrices $M$ (i.e., $M$ is not a block matrix built from the one from Example \ref{ex:1}) that satisfy our assumptions, together with an explicit splitting of the lattice $\Gamma_{P, \tau}$ in a non-trivial case. Take the matrix 
    \[M = \begin{pmatrix}
            5 & 1 & 3 \\
            3 & 1 & 2 \\
            -4 & -1 & -2 \\
    \end{pmatrix}\]
    in $\SL(3, \mathbb{Z})$. $M$ is diagonalizable as $PMP^{-1} = D$ with
    \[P = \begin{pmatrix}
            \frac{-\sqrt{5}-3}{2} & \frac{-\sqrt{5}}{5} & \frac{-2\sqrt{5}-5}{5} \\
            \frac{\sqrt{5}-3}{2} & \frac{\sqrt{5}}{5}& \frac{2\sqrt{5}-5}{5}\\
            3 & 0 & 3 \\
    \end{pmatrix} \qquad D = \begin{pmatrix}
            \frac{\sqrt{5}+3}{2} & 0 & 0 \\
            0 & \frac{-\sqrt{5}+3}{2} & 0 \\
            0 & 0 & 1 \\
    \end{pmatrix}.\]
    We set $N_{4} = \Gamma_{P, \tau} \backslash (\mathbb{C} \ltimes_{\rho} \mathbb{C}^{3})$. The lattice $\Gamma_{P, \tau}$ is then generated by the following transformations of $\C^{4}$, with coordinates $(w, z_{1}, z_{2}, z_{3})$:

    \begin{itemize}
        \item 
        $g_{0} = (w+1,\frac{\sqrt{5}+3}{2}z_{1}, \frac{-\sqrt{5}+3}{2}z_{2}, z_{3} ), h_{0} = (w+i\tau,z_{1}, z_{2}, z_{3} )$\\
        \item $g_{1} = (w,z_{1}+\frac{-\sqrt{5}-3}{2}, z_{2}+\frac{\sqrt{5}-3}{2}, z_{3}+3 ), h_{1} = (w,z_{1}+i\frac{-\sqrt{5}-3}{2}, z_{2}+i\frac{\sqrt{5}-3}{2}, z_{3}+3i )$\\
        \item $g_{2} = (w,z_{1}+\frac{-\sqrt{5}}{5}, z_{2}+\frac{\sqrt{5}}{5}, z_{3} ), h_{2} = (w,z_{1}+i\frac{-\sqrt{5}}{5}, z_{2}+i\frac{\sqrt{5}}{5}, z_{3} )$ \\
        \item $g_{3} = (w,z_{1}+\frac{-2\sqrt{5}-5}{5} , z_{2}+\frac{2\sqrt{5}-5}{5} , z_{3}+3 ), h_{3} = (w,z_{1}+i\frac{-2\sqrt{5}-5}{5} , z_{2}+i\frac{2\sqrt{5}-5}{5} , z_{3}+3i )$\\
     \end{itemize}
 Now we set:
    \begin{itemize}
        \item $g'_{0} = g_{0} = (w+1,\frac{\sqrt{5}+3}{2}z_{1}, \frac{-\sqrt{5}+3}{2}z_{2}, z_{3} ), h'_{0} = h_{0} = (w+i\tau,z_{1}, z_{2}, z_{3} )$ \\
        \item $g'_{1} = g_{3}g_{2}g_{1}^{-1} = (w,z_{1}+\frac{-\sqrt{5}+5}{10}, z_{2}+\frac{\sqrt{5}+5}{10}, z_{3}), h'_{1} = h_{3}h_{2}h_{1}^{-1} = (w,z_{1}+i\frac{-\sqrt{5}+5}{10}, z_{2}+i\frac{\sqrt{5}+5}{10}, z_{3})$\\
        \item $g'_{2} = g_{2} = (w,z_{1}+\frac{-\sqrt{5}}{5}, z_{2}+\frac{\sqrt{5}}{5}, z_{3} ), h'_{2} = h_{2} = (w,z_{1}+i\frac{-\sqrt{5}}{5}, z_{2}+i\frac{\sqrt{5}}{5}, z_{3} )$ \\
        \item $g'_{3} = g_{3}^{3}g_{2}^{-1}g_{1}^{-2} = (w,z_{1} , z_{2} , z_{3}+3 ), h'_{3} = h_{3}^{3}h_{2}^{-1}h_{1}^{-2}= (w,z_{1} , z_{2} , z_{3}+3i ).$\\
    \end{itemize}
    Then $\{g'_{0}, g'_{1}, g'_{2}, g'_{3}, h'_{0}, h'_{1}, h'_{2}, h'_{3}\}$ is a set of generators for $\Gamma_{P, \tau}$ such that the groups $$\Gamma_{P, \tau} \simeq \Gamma_{Q, \tau} \times (3\mathbb{Z} \oplus 3i\mathbb{Z})$$ are isomorphic. Since Theorem \ref{MostowFundGroups} holds, we get 
    \[N_{4} \simeq N_{3} \times \mathbb{T}_{\C}^{1}.\]
The generators for the de Rham cohomology and the Betti numbers of $N_{4}$ are shown in table \ref{table:2}.
\end{esempio}

\begin{esempio}\label{ex:3}
    Taking the matrix $M$ from Example \eqref{ex:1}, we can build quotients of $\C^{2} \ltimes_{\rho} \C^{2}$. We take a lattice $\Gamma_{P, \underline{\mu}}$ as in Theorem \ref{LatticeInGenNakamura} and we set $N_{2,2} := N_{\underline{\mu},P}$. We then take the symplectic form 
    \[\omega = \dfrac{i}{2}\bigl(C_{1}\varphi^{1} \wedge \bar\varphi^{1} + C_{2}\varphi^{2} \wedge \bar\varphi^{2}\bigr) + \dfrac{1}{2}\bigl(B\psi^{1} \wedge \bar\psi^{2} + \bar B\bar\psi^{1} \wedge \psi^{2}\bigr)\]
    with $C_{1}, C_{2} \in \mathbb{R} \setminus \{0\}, B \in \C \setminus \{0\}$. We recall that $N_{2,2}$ is not of completely solvable type, so the de Rham cohomology is not the invariant one. Setting $F := F(w_{1}, w_{2}) = e^{\lambda(w_{1} - \bar w_{1} - w_{2} + \bar{w_{2}})}$ we get generators and Betti numbers for $N_{2,2}$ summarized in table \ref{table:3}. We also report the upper half of the Hodge diamond of $N_{2,2}$ in figure \ref{fig:1}. We point out that, even if the Betti numbers are the same as those of $N_{4}$, the two solvmanifolds have substantial different features: they carry different complex structures, their Hodge numbers differ, $N_{4}$ is of completely solvable type while $N_{2,2}$ is not, $N_{4}$ satisfies the $\del\delbar$-lemma while $N_{2,2}$ does not.\\
    We also provide an example of a non-invariant symplectic structure which also satisfies HLC. We set
    \[\Omega = \dfrac{i}{2}(\varphi^{1\bar1} + \varphi^{2\bar2}) + \dfrac{1}{2}(F\psi^{12} + \bar F\psi^{\bar1\bar2})\]
    where we make use of the notation $\varphi^{ij} = \varphi^{i} \wedge \varphi^{j}$. Since $F\bar F = 1$, we get for the powers of $\Omega$:
    \[\Omega^{2} = -\dfrac{1}{2}\varphi^{1\bar12\bar2} + \dfrac{i}{2}F\varphi^{1\bar1}\psi^{12} + \dfrac{i}{2}\bar F\varphi^{1\bar1}\psi^{\bar1\bar2} + \dfrac{i}{2}F\varphi^{2\bar2}\psi^{12} + \dfrac{i}{2}\bar F\varphi^{2\bar2}\psi^{\bar1\bar2} + \dfrac{1}{2}\psi^{12\bar1\bar2}\]
    \[\Omega^{3} = -\dfrac{3}{2}F\varphi^{1\bar12\bar2}\psi^{12} -\dfrac{3}{2}\bar F\varphi^{1\bar12\bar2}\psi^{\bar1\bar2} +\dfrac{3}{2}i\varphi^{1\bar1}\psi^{12\bar1\bar2} +\dfrac{3}{2}i\varphi^{2\bar2}\psi^{12\bar1\bar2}.\]\\
    We now show that $\Omega$ satisfies HLC. Indeed, we get:
    \begin{itemize}
        \item $L: H^{3}_{d}(N_{2,2}) \rightarrow H^{5}_{d}(N_{2,2})$:
        \begin{gather*}
            L\varphi^{1\bar1 2} =\dfrac{1}{2}F\varphi^{1\bar1 2}\psi^{12} + \dfrac{1}{2}\bar F\varphi^{1\bar1 2}\psi^{\bar1\bar2} \quad 
            L\varphi^{1\bar1\bar2} = \dfrac{1}{2}F\varphi^{1\bar1 \bar2}\psi^{12} + \dfrac{1}{2}\bar F\varphi^{1\bar1 \bar2}\psi^{\bar1\bar2} \\
            L\varphi^{12\bar2} = \dfrac{1}{2}F\varphi^{12\bar2}\psi^{12} + \dfrac{1}{2}\bar F\varphi^{12\bar2}\psi^{\bar1\bar2} \quad
            L\varphi^{\bar1 2\bar2}  = \dfrac{1}{2}F\varphi^{\bar1 2\bar2}\psi^{12} + \dfrac{1}{2}\bar F\varphi^{\bar1 2\bar2}\psi^{\bar1\bar2}\\ 
            LF\varphi^{1}\psi^{12} = \dfrac{i}{2}F\varphi^{12\bar2}\psi^{12} + \dfrac{1}{2}\varphi^{1}\psi^{12\bar1\bar2} \quad LF\varphi^{\bar1}\psi^{12} = \dfrac{i}{2}F\varphi^{\bar1 2\bar2}\psi^{12} + \dfrac{1}{2}\varphi^{\bar1}\psi^{12\bar1\bar2}\\
            LF\varphi^{2}\psi^{12} = \dfrac{i}{2}F\varphi^{1\bar1 2}\psi^{12} + \dfrac{1}{2}\varphi^{2}\psi^{12\bar1\bar2} \quad LF\varphi^{\bar2}\psi^{12} = \dfrac{i}{2}F\varphi^{1\bar1 \bar2}\psi^{12} + \dfrac{1}{2}\varphi^{\bar2}\psi^{12\bar1\bar2} \\
            L\varphi^{1}\psi^{\bar1 2} = \dfrac{i}{2}\varphi^{12\bar2}\psi^{\bar1 2} \quad
            L\varphi^{\bar1}\psi^{\bar1 2} = \dfrac{i}{2}\varphi^{\bar1 2\bar2}\psi^{\bar1 2} \quad
            L\varphi^{2}\psi^{\bar1 2} = \dfrac{i}{2}\varphi^{1\bar1 2}\psi^{\bar1 2} \quad 
            L\varphi^{\bar2}\psi^{\bar1 2} = \dfrac{i}{2}\varphi^{1\bar1\bar2}\psi^{\bar1 2}\\
            L\varphi^{1}\psi^{1 \bar2} = \dfrac{i}{2}\varphi^{12\bar2}\psi^{1 \bar2} \quad L\varphi^{\bar1}\psi^{1 \bar2} = \dfrac{i}{2}\varphi^{\bar1 2\bar2}\psi^{1 \bar2} \quad
            L\varphi^{2}\psi^{1 \bar2} = \dfrac{i}{2}\varphi^{1\bar1 2}\psi^{1 \bar2} \quad
            L\varphi^{\bar2}\psi^{1 \bar2} = \dfrac{i}{2}\varphi^{1\bar1 \bar2}\psi^{1 \bar2} \\
            L\bar F\varphi^{1}\psi^{\bar1\bar2} = \dfrac{i}{2}\bar F\varphi^{12\bar2}\psi^{\bar1\bar2} + \dfrac{1}{2} \varphi^{1}\psi^{12\bar1\bar2} \quad 
            L\bar F\varphi^{\bar1}\psi^{\bar1\bar2} = \dfrac{i}{2}\bar F\varphi^{\bar1 2\bar2}\psi^{\bar1\bar2} + \dfrac{1}{2} \varphi^{\bar1}\psi^{12\bar1\bar2} \\
            L\bar F\varphi^{2}\psi^{\bar1\bar2} = \dfrac{i}{2}\bar F\varphi^{1\bar1 2}\psi^{\bar1\bar2} + \dfrac{1}{2} \varphi^{2}\psi^{12\bar1\bar2} \quad 
            L\bar F\varphi^{\bar2}\psi^{\bar1\bar2} = \dfrac{i}{2}\bar F\varphi^{1\bar1\bar2}\psi^{\bar1\bar2} + \dfrac{1}{2} \varphi^{\bar2}\psi^{12\bar1\bar2}. \\
        \end{gather*}
        \item $L^{2}: H^{2}_{d}(N_{2,2}) \rightarrow H^{6}_{d}(N_{2,2})$:
        \begin{gather*}
            L^{2}\varphi^{1\bar1} = \dfrac{i}{2}F\varphi^{1\bar12\bar2}\psi^{12} + \dfrac{i}{2}\bar F\varphi^{1\bar12\bar2}\psi^{\bar1\bar2} + \dfrac{1}{2}\varphi^{1\bar1}\psi^{12\bar1\bar2}\\
            L^{2}\varphi^{1\bar2} = \dfrac{1}{2}\varphi^{1\bar2}\psi^{12\bar1\bar2} \quad 
            L^{2}\varphi^{\bar1 2} \dfrac{1}{2}\varphi^{\bar1 2}\psi^{12\bar1\bar2} \\
            L^{2}\varphi^{1 2} = \dfrac{1}{2}\varphi^{1 2}\psi^{12\bar1\bar2} \quad
            L^{2}\varphi^{\bar1\bar2} = \dfrac{1}{2}\varphi^{\bar1 \bar2}\psi^{12\bar1\bar2} \\
            L^{2}\varphi^{2\bar2} = \dfrac{i}{2}F\varphi^{1\bar12\bar2}\psi^{12} + \dfrac{i}{2}\bar F\varphi^{1\bar12\bar2}\psi^{\bar1\bar2} + \dfrac{1}{2}\varphi^{2\bar2}\psi^{12\bar1\bar2} \\
            L^{2}\psi^{\bar1 2} = -\dfrac{1}{2}\varphi^{1\bar12\bar2}\psi^{\bar1 2}\quad
            L^{2}\psi^{1 \bar2} = -\dfrac{1}{2}\varphi^{1\bar12\bar2}\psi^{1 \bar2}  \\
            L^{2}F\psi^{1 2} =  -\dfrac{1}{2}\varphi^{1\bar12\bar2}\psi^{1 2} + \dfrac{i}{2}\bar F \varphi^{1\bar1}\psi^{12\bar1\bar2} + \dfrac{i}{2}\bar F \varphi^{2\bar2}\psi^{12\bar1\bar2}\\
            L^{2}F\psi^{\bar1\bar2} = -\dfrac{1}{2}\varphi^{1\bar12\bar2}\psi^{\bar1\bar2} + \dfrac{i}{2}F \varphi^{1\bar1}\psi^{12\bar1\bar2} + \dfrac{i}{2}F \varphi^{2\bar2}\psi^{12\bar1\bar2}. \\
        \end{gather*}
        \item $L^{3}: H^{1}_{d}(N_{2,2}) \rightarrow H^{7}_{d}(N_{2,2})$:
        \begin{gather*}
            L^{3}\varphi^{1} = -\dfrac{3}{2}i\varphi^{12\bar2}\psi^{12\bar1\bar2}
            \quad  L^{3}\varphi^{\bar1} = -\dfrac{3}{2}i\varphi^{\bar1 2\bar2}\psi^{12\bar1\bar2}\\
            L^{3}\varphi^{2} = \dfrac{3}{2}i\varphi^{1\bar1 2}\psi^{12\bar1\bar2}
            \quad  L^{3}\varphi^{\bar2} = \dfrac{3}{2}i\varphi^{1\bar1 \bar2}\psi^{12\bar1\bar2}. \\
        \end{gather*}
    \end{itemize}
    
\end{esempio}

\begin{esempio}
    For the construction of the symplectic manifolds from $G_{n,m}$, we made an assumption over the matrix $M$: the matrix $D$ had to be in the form $D = \hbox{diag}(e^{\lambda}, e^{-\lambda}, \ldots, e^{\lambda}, e^{-\lambda})$. We now just want to show that there exist non-trivial matrices $M \in \SL(4,\mathbb{Z})$ (i.e., not block matrices built from the one of Example \ref{ex:1}) that satisfy this condition. Take the matrix
    \[M = \begin{pmatrix}
        2 & 3 & -3 & 0 \\
        -10 & 14 & -12 & 3 \\
        -11 & 12 & -10 & 3 \\
        -4 & 11 & -10 & 2 \\
    \end{pmatrix}.\]
    Then $M \in \SL(4,\mathbb{Z})$ and it is similar to the diagonal matrix $D = diag(\sqrt{3} + 2, -\sqrt{3} + 2, \sqrt{3} + 2, -\sqrt{3} + 2)$, which is the requested form and allows the building of the symplectic manifolds from $\C^{2n} \ltimes_{\rho} \C^{4}$. It also can be used to produce generalized Nakamura manifolds of complex dimension five, for which we computed the almost-complex Kodaira dimension in section 4.
\end{esempio}

 \begin{table}[H]
	\centering
	\begin{tabular}{|c|c|c|}
	\hline
	$ k $ & $ H_{d}^{k}(N_{3}, \C) $ & $ b_{k}(N_{3}) $ \\
	\hline
	0 & 1 & 1 \\
	\hline
	1 & $ \varphi^{0}, \bar\varphi^{0} $ & 2 \\
	\hline
	2 & $\varphi^{0}\bar\varphi^{0}, \varphi^{1}\varphi^{2}, \bar\varphi^{1}\varphi^{2}, \varphi^{1}\bar\varphi^{2}, \bar\varphi^{1}\bar\varphi^{2} $ & 5 \\
	\hline
	3 & $ \varphi^{0}\varphi^{1}\varphi^{2}, \varphi^{0}\bar\varphi^{1}\varphi^{2}, \varphi^{0}\varphi^{1}\bar\varphi^{2}, \varphi^{0}\bar\varphi^{1}\bar\varphi^{2}, \bar\varphi^{0}\varphi^{1}\varphi^{2}, \bar\varphi^{0}\bar\varphi^{1}\varphi^{2}, \bar\varphi^{0}\varphi^{1}\bar\varphi^{2}, \bar\varphi^{0}\bar\varphi^{1}\bar\varphi^{2}$ & 8 \\
	\hline
	4 & $ \varphi^{0}\bar\varphi^{0}\varphi^{1}\varphi^{2}, \varphi^{0}\bar\varphi^{0}\bar\varphi^{1}\varphi^{2}, \varphi^{0}\bar\varphi^{0}\varphi^{1}\bar\varphi^{2}, \varphi^{0}\bar\varphi^{0}\bar\varphi^{1}\bar\varphi^{2}, \varphi^{1}\bar\varphi^{1}\varphi^{2}\bar\varphi^{2}$ & 5 \\
	\hline
	5 & $ \varphi^{0}\varphi^{1}\bar\varphi^{1}\varphi^{2}\bar\varphi^{2}, \bar\varphi^{0}\varphi^{1}\bar\varphi^{1}\varphi^{2}\bar\varphi^{2}$ & 2 \\
	\hline
	6 & $ \omega^{3}$ & 1 \\
	\hline
	\end{tabular}
    \caption{de Rham cohomology and Betti numbers for $ N_{3} $}
    \label{table:1}
    \end{table}

    \begin{figure}[H]
        \centering
        \begin{tikzcd}[row sep=small, column sep=tiny]
            p+q = 0 &&&&&& 1 \\
            p+q = 1 &&&&& 4 && 4 \\
            p+q = 2 &&&& 6 && 8 && 6 \\
            p+q = 3 &&& 4 && 12 && 12 && 4 \\
            p+q = 4 && 1 && 16 && 36 && 16 && 1 \\
        \end{tikzcd}
        \caption{Hodge numbers for $N_{2,2}$}
        \label{fig:1}
    \end{figure}

    \newpage
    \begin{table}[H]
	\centering
	\begin{tabular}{|c|c|c|}
	\hline
	$ k $ & $ H_{d}^{k}(N_{4}, \C) $ & $ b_{k}(N_{4}) $ \\
	\hline
	0 & 1 & 1 \\
	\hline
	1 & $ \varphi^{0}, \bar\varphi^{0}, \varphi^{3}, \bar\varphi^{3} $ & 4 \\
	\hline
	2 & $\varphi^{0}\bar\varphi^{0}, \varphi^{0}\varphi^{3},\varphi^{0}\bar\varphi^{3}, \varphi^{3}\bar\varphi^{0},\bar\varphi^{0}\bar\varphi^{3}, \varphi^{3}\bar\varphi^{3}, \varphi^{1}\varphi^{2}, \bar\varphi^{1}\varphi^{2}, \varphi^{1}\bar\varphi^{2}, \bar\varphi^{1}\bar\varphi^{2} $ & 10 \\
	\hline
	\multirow{3}{4em}{\centering3} & $ \varphi^{0}\bar\varphi^{0}\varphi^{3}, \varphi^{0}\bar\varphi^{0}\bar\varphi^{3}, \varphi^{0}\varphi^{3}\bar\varphi^{3}, \bar\varphi^{0}\varphi^{3}\bar\varphi^{3}, \varphi^{0}\varphi^{1}\varphi^{2}, \varphi^{0}\bar\varphi^{1}\varphi^{2}, \varphi^{0}\varphi^{1}\bar\varphi^{2}, \varphi^{0}\bar\varphi^{1}\bar\varphi^{2}$& \multirow{3}{4em}{\centering20} \\
    & $\bar\varphi^{0}\varphi^{1}\varphi^{2}, \bar\varphi^{0}\bar\varphi^{1}\varphi^{2}, \bar\varphi^{0}\varphi^{1}\bar\varphi^{2}, \bar\varphi^{0}\bar\varphi^{1}\bar\varphi^{2}, \varphi^{3}\varphi^{1}\varphi^{2}, \varphi^{3}\bar\varphi^{1}\varphi^{2} $ & \\
    &$\varphi^{3}\varphi^{1}\bar\varphi^{2}, \varphi^{3}\bar\varphi^{1}\bar\varphi^{2}, \bar\varphi^{3}\varphi^{1}\varphi^{2}, \bar\varphi^{3}\bar\varphi^{1}\varphi^{2}, \bar\varphi^{3}\varphi^{1}\bar\varphi^{2}, \bar\varphi^{3}\bar\varphi^{1}\bar\varphi^{2}$ & \\
	\hline
    \multirow{5}{4em}{\centering4} & $ \varphi^{0}\bar\varphi^{0}\varphi^{1}\varphi^{2}, \varphi^{0}\bar\varphi^{0}\bar\varphi^{1}\varphi^{2}, \varphi^{0}\bar\varphi^{0}\varphi^{1}\bar\varphi^{2}, \varphi^{0}\bar\varphi^{0}\bar\varphi^{1}\bar\varphi^{2}, \varphi^{1}\bar\varphi^{1}\varphi^{2}\bar\varphi^{2}, \varphi^{3}\bar\varphi^{3}\varphi^{1}\varphi^{2}$ &  \multirow{5}{4em}{\centering26}\\
    & $\varphi^{3}\bar\varphi^{3}\bar\varphi^{1}\varphi^{2}, \varphi^{3}\bar\varphi^{3}\varphi^{1}\bar\varphi^{2}, \varphi^{3}\bar\varphi^{3}\bar\varphi^{1}\bar\varphi^{2},\varphi^{0}\varphi^{3}\varphi^{1}\varphi^{2}, \varphi^{0}\varphi^{3}\varphi^{1}\bar\varphi^{2}$ &\\
    & $\varphi^{0}\varphi^{3}\bar\varphi^{1}\varphi^{2}, \varphi^{0}\varphi^{3}\bar\varphi^{1}\bar\varphi^{2}, \varphi^{0}\bar\varphi^{3}\varphi^{1}\varphi^{2}, \varphi^{0}\bar\varphi^{3}\varphi^{1}\bar\varphi^{2}, \varphi^{0}\bar\varphi^{3}\bar\varphi^{1}\varphi^{2}$ & \\
    &$\varphi^{0}\bar\varphi^{3}\bar\varphi^{1}\bar\varphi^{2}, \bar\varphi^{0}\varphi^{3}\varphi^{1}\varphi^{2}, \bar\varphi^{0}\varphi^{3}\varphi^{1}\bar\varphi^{2}, \bar\varphi^{0}\varphi^{3}\bar\varphi^{1}\varphi^{2}, \bar\varphi^{0}\varphi^{3}\bar\varphi^{1}\bar\varphi^{2}$& \\
    &$\bar\varphi^{0}\bar\varphi^{3}\varphi^{1}\varphi^{2}, \bar\varphi^{0}\bar\varphi^{3}\varphi^{1}\bar\varphi^{2}, \bar\varphi^{0}\bar\varphi^{3}\bar\varphi^{1}\varphi^{2}, \bar\varphi^{0}\bar\varphi^{3}\bar\varphi^{1}\bar\varphi^{2}, \varphi^{0}\bar\varphi^{0}\varphi^{3}\bar\varphi^{3}$& \\
	\hline
	\multirow{5}{4em}{\centering5} & $ \varphi^{0}\bar\varphi^{0}\varphi^{3}\varphi^{1}\varphi^{2}, \varphi^{0}\bar\varphi^{0}\varphi^{3}\bar\varphi^{1}\varphi^{2}, \varphi^{0}\bar\varphi^{0}\varphi^{3}\varphi^{1}\bar\varphi^{2}, \varphi^{0}\bar\varphi^{0}\varphi^{3}\bar\varphi^{1}\bar\varphi^{2}$& \multirow{5}{4em}{\centering20} \\
    & $\varphi^{0}\bar\varphi^{0}\bar\varphi^{3}\varphi^{1}\varphi^{2}, \varphi^{0}\bar\varphi^{0}\bar\varphi^{3}\bar\varphi^{1}\varphi^{2}, \varphi^{0}\bar\varphi^{0}\bar\varphi^{3}\varphi^{1}\bar\varphi^{2}, \varphi^{0}\bar\varphi^{0}\bar\varphi^{3}\bar\varphi^{1}\bar\varphi^{2}$ & \\
    &$\varphi^{0}\varphi^{3}\bar\varphi^{3}\varphi^{1}\varphi^{2}, \varphi^{0}\varphi^{3}\bar\varphi^{3}\bar\varphi^{1}\varphi^{2}, \varphi^{0}\varphi^{3}\bar\varphi^{3}\varphi^{1}\bar\varphi^{2}, \varphi^{0}\varphi^{3}\bar\varphi^{3}\bar\varphi^{1}\bar\varphi^{2}$&\\
    &$,\bar\varphi^{0}\varphi^{3}\bar\varphi^{3}\varphi^{1}\varphi^{2} \bar\varphi^{0}\varphi^{3}\bar\varphi^{3}\bar\varphi^{1}\varphi^{2}, \bar\varphi^{0}\varphi^{3}\bar\varphi^{3}\varphi^{1}\bar\varphi^{2}, \bar\varphi^{0}\varphi^{3}\bar\varphi^{3}\bar\varphi^{1}\bar\varphi^{2}$ & \\
    &$\varphi^{0}\varphi^{1}\bar\varphi^{1}\varphi^{2}\bar\varphi^{2}, \bar\varphi^{0}\varphi^{1}\bar\varphi^{1}\varphi^{2}\bar\varphi^{2},\varphi^{3}\varphi^{1}\bar\varphi^{1}\varphi^{2}\bar\varphi^{2},\bar\varphi^{3}\varphi^{1}\bar\varphi^{1}\varphi^{2}\bar\varphi^{2}$&\\
	\hline
	\multirow{3}{4em}{\centering6} & $\varphi^{0}\bar\varphi^{0}\varphi^{1}\bar\varphi^{1}\varphi^{2}\bar\varphi^{2}, \varphi^{0}\varphi^{3}\varphi^{1}\bar\varphi^{1}\varphi^{2}\bar\varphi^{2},\varphi^{0}\bar\varphi^{3}\varphi^{1}\bar\varphi^{1}\varphi^{2}\bar\varphi^{2}, \varphi^{3}\bar\varphi^{0}\varphi^{1}\bar\varphi^{1}\varphi^{2}\bar\varphi^{2} $ & \multirow{3}{4em}{\centering10} \\ &$\bar\varphi^{0}\bar\varphi^{3}\varphi^{1}\bar\varphi^{1}\varphi^{2}\bar\varphi^{2},\varphi^{3}\bar\varphi^{3}\varphi^{1}\bar\varphi^{1}\varphi^{2}\bar\varphi^{2}, \varphi^{0}\bar\varphi^{0}\varphi^{3}\bar\varphi^{3}\varphi^{1}\varphi^{2}$&\\
    &$\varphi^{0}\bar\varphi^{0}\varphi^{3}\bar\varphi^{3}\varphi^{1}\bar\varphi^{2}, \varphi^{0}\bar\varphi^{0}\varphi^{3}\bar\varphi^{3}\bar\varphi^{1}\varphi^{2}, \varphi^{0}\bar\varphi^{0}\varphi^{3}\bar\varphi^{3}\bar\varphi^{1}\bar\varphi^{2}$&\\
	\hline
    7 &$\varphi^{0}\bar\varphi^{0}\varphi^{3}\varphi^{1}\bar\varphi^{1}\varphi^{2}\bar\varphi^{2}, \varphi^{0}\bar\varphi^{0}\bar\varphi^{3}\varphi^{1}\bar\varphi^{1}\varphi^{2}\bar\varphi^{2}, \varphi^{0}\varphi^{3}\bar\varphi^{3}\varphi^{1}\bar\varphi^{1}\varphi^{2}\bar\varphi^{2}, \bar\varphi^{0}\varphi^{3}\bar\varphi^{3}\varphi^{1}\bar\varphi^{1}\varphi^{2}\bar\varphi^{2} $& 4\\
    \hline
    8 & $\omega^{4}$ & 1\\
    \hline
	\end{tabular}
    \caption{de Rham cohomology and Betti numbers for $ N_{4} $}
    \label{table:2}
    \end{table}

   \begin{table}[H]
	\centering
	\begin{tabular}{|c|c|c|}
	\hline
	$ k $ & $ H_{d}^{k}(N_{2,2}, \C) $ & $ b_{k}(N_{2,2}) $ \\
	\hline
	0 & 1 & 1 \\
	\hline
	1 & $ \varphi^{1}, \bar\varphi^{1}, \varphi^{2}, \bar\varphi^{2}$ & 4 \\
	\hline
	2 & $\varphi^{1}\bar\varphi^{1}, \varphi^{1}\bar\varphi^{2}, \bar\varphi^{1}\varphi^{2}, \varphi^{1}\varphi^{2}, \bar\varphi^{1}\bar\varphi^{2}, \varphi^{2}\bar\varphi^{2}, \bar\psi^{1}\psi^{2}, \psi^{1}\bar\psi^{2},F\psi^{1}\psi^{2}, \bar F\bar\psi^{1}\bar\psi^{2} $ & 10 \\
	\hline
	\multirow{3}{4em}{\centering3} & $ \varphi^{1}\bar\varphi^{1}\varphi^{2}, \varphi^{1}\bar\varphi^{1}\bar\varphi^{2}, \varphi^{1}\varphi^{2}\bar\varphi^{2}, \bar\varphi^{1}\varphi^{2}\bar\varphi^{2}, F\varphi^{1}\psi^{1}\psi^{2},F\bar\varphi^{1}\psi^{1}\psi^{2},F\varphi^{2}\psi^{1}\psi^{2}$& \multirow{3}{4em}{\centering20} \\
    & $F\bar\varphi^{2}\psi^{1}\psi^{2}, \varphi^{1}\bar\psi^{1}\psi^{2},\bar\varphi^{1}\bar\psi^{1}\psi^{2},\varphi^{2}\bar\psi^{1}\psi^{2},\bar\varphi^{2}\bar\psi^{1}\psi^{2},\varphi^{1}\psi^{1}\bar\psi^{2}$ & \\
    &$\bar\varphi^{1}\psi^{1}\bar\psi^{2},\varphi^{2}\psi^{1}\bar\psi^{2},\bar\varphi^{2}\psi^{1}\bar\psi^{2} \bar F\varphi^{1}\bar\psi^{1}\bar\psi^{2},\bar F\bar\varphi^{1}\bar\psi^{1}\bar\psi^{2},\bar F\varphi^{2}\bar\psi^{1}\bar\psi^{2}, \bar F\bar\varphi^{2}\bar\psi^{1}\bar\psi^{2}$ & \\
	\hline
    \multirow{4}{4em}{\centering4} & $ F\varphi^{1}\bar\varphi^{1}\psi^{1}\psi^{2}, F\varphi^{1}\varphi^{2}\psi^{1}\psi^{2},F\varphi^{1}\bar\varphi^{2}\psi^{1}\psi^{2},F\bar\varphi^{1}\varphi^{2}\psi^{1}\psi^{2}, F\bar\varphi^{1}\bar\varphi^{2}\psi^{1}\psi^{2}$ &  \multirow{4}{4em}{\centering26}\\
    &$F\varphi^{2}\bar\varphi^{2}\psi^{1}\psi^{2},\bar F\varphi^{1}\bar\varphi^{1}\bar\psi^{1}\bar\psi^{2}, \bar F\varphi^{1}\varphi^{2}\bar\psi^{1}\bar\psi^{2},\bar F\varphi^{1}\bar\varphi^{2}\bar\psi^{1}\bar\psi^{2},\bar F\bar\varphi^{1}\varphi^{2}\bar\psi^{1}\bar\psi^{2}$&\\
    & $\bar F\bar\varphi^{1}\bar\varphi^{2}\bar\psi^{1}\bar\psi^{2},\bar F\varphi^{2}\bar\varphi^{2}\bar\psi^{1}\bar\psi^{2},\varphi^{1}\bar\varphi^{1}\bar\psi^{1}\psi^{2}, \varphi^{1}\varphi^{2}\bar\psi^{1}\psi^{2},\varphi^{1}\bar\varphi^{2}\bar\psi^{1}\psi^{2}$ &\\
    & $\bar\varphi^{1}\varphi^{2}\bar\psi^{1}\psi^{2}, \bar\varphi^{1}\bar\varphi^{2}\bar\psi^{1}\psi^{2},\varphi^{2}\bar\varphi^{2}\bar\psi^{1}\psi^{2}, \psi^{1}\bar\psi^{1}\psi^{2}\bar\psi^{2},\varphi^{1}\bar\varphi^{1}\psi^{1}\bar\psi^{2}$ & \\ &$\varphi^{1}\varphi^{2}\psi^{1}\bar\psi^{2},\varphi^{1}\bar\varphi^{2}\psi^{1}\bar\psi^{2},\bar\varphi^{1}\varphi^{2}\psi^{1}\bar\psi^{2}, \bar\varphi^{1}\bar\varphi^{2}\psi^{1}\bar\psi^{2},\varphi^{2}\bar\varphi^{2}\psi^{1}\bar\psi^{2}, \varphi^{1}\bar\varphi^{1}\varphi^{2}\bar\varphi^{2}  $& \\
	\hline
	\multirow{4}{4em}{\centering5} & $F\varphi^{1}\bar\varphi^{1}\varphi^{2}\psi^{1}\psi^{2}, F\varphi^{1}\bar\varphi^{1}\bar\varphi^{2}\psi^{1}\psi^{2}, F\varphi^{1}\varphi^{2}\bar\varphi^{2}\psi^{1}\psi^{2}, F\bar\varphi^{1}\varphi^{2}\bar\varphi^{2}\psi^{1}\psi^{2}, \varphi^{1}\bar\varphi^{1}\varphi^{2}\psi^{1}\bar\psi^{2} $ & \multirow{4}{4em}{\centering20} \\
    & $ \varphi^{1}\varphi^{2}\bar\varphi^{2}\psi^{1}\bar\psi^{2} ,\varphi^{1}\bar\varphi^{1}\bar\varphi^{2}\psi^{1}\bar\psi^{2}, \bar\varphi^{1}\varphi^{2}\bar\varphi^{2}\psi^{1}\bar\psi^{2},  \varphi^{1}\bar\varphi^{1}\varphi^{2}\bar\psi^{1}\psi^{2},\varphi^{1}\varphi^{2}\bar\varphi^{2}\bar\psi^{1}\psi^{2}    $ & \\
    &$\varphi^{1}\bar\varphi^{1}\bar\varphi^{2}\bar\psi^{1}\psi^{2}, \bar\varphi^{1}\varphi^{2}\bar\varphi^{2}\bar\psi^{1}\psi^{2}, \bar F\varphi^{1}\bar\varphi^{1}\varphi^{2}\bar\psi^{1}\bar\psi^{2}, \bar F\varphi^{1}\bar\varphi^{1}\bar\varphi^{2}\bar\psi^{1}\bar\psi^{2}, \bar F\varphi^{1}\varphi^{2}\bar\varphi^{2}\bar \psi^{1}\bar \psi^{2} $ & \\
    &$ \bar F\bar\varphi^{1}\varphi^{2}\bar\varphi^{2}\psi^{1}\bar\psi^{1}\bar\psi^{2},\varphi^{1}\psi^{1}\bar\psi^{1}\psi^{2}\bar\psi^{2}, \bar\varphi^{1}\psi^{1}\bar\psi^{1}\psi^{2}\bar\psi^{2}, \varphi^{2}\psi^{1}\bar\psi^{1}\psi^{2}\bar\psi^{2}, \bar\varphi^{2}\psi^{1}\bar\psi^{1}\psi^{2}\bar\psi^{2}$&\\
	\hline
	\multirow{2}{4em}{\centering6} & $ \varphi^{1}\bar\varphi^{1}\psi^{1}\bar\psi^{1}\psi^{2}\bar\psi^{2}, \varphi^{2}\bar\varphi^{2}\psi^{1}\bar\psi^{1}\psi^{2}\bar\psi^{2}, \varphi^{1}\bar\varphi^{2}\psi^{1}\bar\psi^{1}\psi^{2}\bar\psi^{2}, \bar\varphi^{1}\varphi^{2}\psi^{1}\bar\psi^{1}\psi^{2}\bar\psi^{2}$ & \multirow{2}{4em}{\centering10} \\
    &$\varphi^{1}\varphi^{2}\psi^{1}\bar\psi^{1}\psi^{2}\bar\psi^{2}, \bar\varphi^{1}\bar\varphi^{2}\psi^{1}\bar\psi^{1}\psi^{2}\bar\psi^{2}, F\varphi^{1}\bar\varphi^{1}\varphi^{2}\bar\varphi^{2}\psi^{1}\psi^{2}, \varphi^{1}\bar\varphi^{1}\varphi^{2}\bar\varphi^{2}\psi^{1}\bar\psi^{2} $&\\
    & $\varphi^{1}\bar\varphi^{1}\varphi^{2}\bar\varphi^{2}\bar\psi^{1}\psi^{2}, \bar F \varphi^{1}\bar\varphi^{1}\varphi^{2}\bar\varphi^{2}\bar\psi^{1}\bar\psi^{2} $ & \\
	\hline
    7 &$ \varphi^{1}\bar\varphi^{1}\varphi^{2}\psi^{1}\bar\psi^{1}\psi^{2}\bar\psi^{2}, \varphi^{1}\bar\varphi^{1}\bar\varphi^{2}\psi^{1}\bar\psi^{1}\psi^{2}\bar\psi^{2}, \varphi^{1}\varphi^{2}\bar\varphi^{2}   \psi^{1}\bar\psi^{1}\psi^{2}\bar\psi^{2}, \bar\varphi^{1}\varphi^{2}\bar\varphi^{2}\psi^{1}\bar\psi^{1}\psi^{2}\bar\psi^{2}$& 4\\
    \hline
    8 & $\omega^{4}$ & 1\\
    \hline
	\end{tabular}
    \caption{de Rham cohomology and Betti numbers for $ N_{2,2} $}
    \label{table:3}
    \end{table}

\end{document}